\documentclass[11pt,reqno]{amsart}

\usepackage[text={150mm,230mm},centering]{geometry}
\usepackage{graphicx}
\usepackage{amssymb}
\usepackage{mathrsfs}
\usepackage[all]{xy}

\usepackage{hyperref}

\newtheorem{theorem}{Theorem}[section]
\newtheorem{lemma}[theorem]{Lemma}

\newtheorem{corollary}[theorem]{Corollary}
\newtheorem{proposition}[theorem]{Proposition}

\theoremstyle{definition}
\newtheorem{example}[theorem]{Example}
\newtheorem{definition}[theorem]{Definition}
\newtheorem{definition-lemma}[theorem]{Definition-Lemma}
\newtheorem{definition-theorem}[theorem]{Definition-Theorem}
\newtheorem{remark}[theorem]{Remark}

\newtheorem*{ack}{Acknowledgements}
\newtheorem*{convention}{Convention}

\newcommand{\ac}{\textup{!`}}

\numberwithin{equation}{section}

\newcommand{\hdot}{\;\raisebox{3.2pt}{\text{\circle*{2.5}}}}

\allowdisplaybreaks

\title[Gravity algebra structure on the negative cyclic homology]
{Gravity algebra structure on the negative cyclic homology of 
Calabi-Yau algebras}

\author[X. Chen]{Xiaojun Chen}

\author[F. Eshmatov]{Farkhod Eshmatov}

\author[L. Liu]{Leilei Liu}

\address{Department of Mathematics, Sichuan University, Chengdu, Sichuan Province 610064 P. R. China}
\email{xjchen@scu.edu.cn, olimjon55@hotmail.com, liuleilei199009@126.com}

\begin{document}

\begin{abstract}
In this paper, we study the gravity algebra structure
on the negative cyclic homology or the cyclic cohomology of several
classes of algebras. These algebras
include: Calabi-Yau algebras, symmetric Frobenius algebras,
unimodular Poisson algebras, and unimodular Frobenius Poisson algebras.
The relationships among these gravity algebras are also discussed
under some additional conditions.
\end{abstract}

\maketitle

\setcounter{tocdepth}{1}

\tableofcontents


\section{Introduction}\label{Sect_intro}

In 1994, Getzler \cite{Getzler2} showed that, 
$\{\mathrm H_\bullet(\mathcal M_{0,n+1})\}$,
the collection of the homology classes of the moduli spaces of 
Riemann spheres with $n+1$ marked points,
forms an operad, which he called the {\it gravity operad},
and that an algebra over it consists
a sequence of skew-symmetric brackets satisfying
the so-called {\it generalized Jacobi identity}, which he called {\it gravity algebra}. 
He later showed in \cite{Getzler3} 
that $\{\mathrm H_\bullet(\mathcal M_{0,n+1})\}$ is Koszul dual, 
in the sense of Ginzburg and Kapranov \cite{GK},
to the operad $\{\mathrm H_\bullet(\overline{\mathcal M}_{0,n+1})\}$, 
the collection of the homology classes of the {\it compactified} moduli spaces of 
Riemann spheres with $n+1$ marked points,
algebras over which are also called hypercommutative algebras or 
formal Frobenius manifolds (see also Manin \cite[Ch. III]{Manin}).

Besides examples from topological field theories
(\cite[Theorem 4.6]{Getzler2}), the first nontrivial example of gravity algebra,
to the authors' best knowledge,
arises from string topology. In \cite{CS},
Chas and Sullivan showed that the $S^1$-equivariant homology
of the free loop space of a smooth compact manifold, after shifting some degree, 
forms a gravity algebra.
More examples, more or less inspired by string topology, can be found in
Westerland \cite{Westerland} and Ward \cite{Ward2} (see also Menichi \cite{Menichi} for some partial results).
These works show that the cyclic cohomology of a symmetric Frobenius (or
more generally cyclic A$_\infty$) algebra,
which is the algebraic model of the $S^1$-equivariant homology of the free loop spaces,
has a nontrivial gravity algebra structure.

A symmetric Frobenius algebra is an associative algebra
with a non-degenerate symmetric pairing. Let us recall two facts on this algebra:
\begin{enumerate}
\item[(i)] if a symmetric Frobenius algebra is Koszul, then its Koszul dual is a Calabi-Yau algebra,
a notion introduced by Ginzburg in \cite{Ginzburg} (see Van den Bergh \cite[Theorem 12.1]{VdB} for a proof);
\item[(ii)] the cyclic cohomology of an algebra is isomorphic to the negative
cyclic homology of its Koszul dual (see \cite[Theorem 37]{CYZ} for a proof).
\end{enumerate}
From these two facts one deduces that there is a gravity algebra structure
on the negative cyclic homology of a Koszul Calabi-Yau algebra (see Ward \cite[Example 6.5]{Ward2}).
It is then naturally expected that for
{arbitrary} Calabi-Yau algebras which are
not necessarily Koszul,
such gravity structure still exists
on their negative cyclic homology.
In loc cit Ward wrote 
``{\it How to directly construct the higher brackets
comprising this [gravity] structure on the Calabi-Yau side is an open question}".

The purpose of this paper is to answer this open question.
To this end, let us first recall that
for a mixed complex $(\mathrm C_\bullet, b, B)$,
if we denote by
 $\mathrm{HH}_\bullet$ and $\mathrm{HC}_\bullet^{-}$
its Hochschild homology (that is,
the homology of $(\mathrm C_\bullet, b)$) 
and its negative cyclic
homology, then we have the long exact sequence
$$
\cdots\longrightarrow
\mathrm{HC}_{\bullet+2}^{-} 
\longrightarrow
\mathrm{HC}_\bullet^{-} 
\stackrel{\pi_*}\longrightarrow
\mathrm{HH}_\bullet 
\stackrel{\beta}\longrightarrow
\mathrm{HC}_{\bullet+1}^{-} 
\longrightarrow\cdots,
$$
The main result of the current paper is
the following slightly technical result:

\begin{theorem}\label{mainthm}
$(1)$ For a mixed $(\mathrm C_\bullet, b, B)$, if $\mathrm{HH}_\bullet$ 
has a Batalin-Vilkovisky algebra structure
such that $B$ is the generator of the Gerstenhaber bracket, then 
the following sequence of maps
$$\begin{array}{ccl}
(\mathrm{HC}^{-}_\bullet)^{\otimes n}&\longrightarrow& \mathrm{HC}^{-}_\bullet\\
(x_1,\cdots, x_n)&\longmapsto&(-1)^{(n-1)|x_1|+(n-2)|x_2|+\cdots+|x_{n-1}|}\beta\big(\pi_*(x_1)
\hdot \pi_*(x_2)\hdot\cdots\hdot
\pi_*(x_n)\big),\end{array}
$$
for $n=2, 3,\cdots$, where $\hdot$ is the product on the Hochschild homology,
give on $\mathrm{HC}^{-}_\bullet$ a gravity algebra structure.

$(2)$ Suppose 
$$f:(\mathrm C_\bullet,b,B)  \longrightarrow (\mathrm C'_\bullet,b',B')$$ 
is a quasi-isomorphism of mixed complexes 
and furthermore suppose that 
$(\mathrm{HH}_\bullet(\mathrm C), B)$ and $(\mathrm{HH}_\bullet(\mathrm C'), B')$ 
are extended to Batalin-Vilkovisky algebras such that $f_*$ 
is a Batalin-Vilkovisky algebra isomorphism. 
Then $\mathrm{HC}_\bullet^{-}(\mathrm C)$ and 
$\mathrm{HC}_\bullet^{-}(\mathrm C')$ are isomorphic as gravity algebras. 
\end{theorem}

The following complexes satisfy the conditions of the above theorem (more details 
of these complexes
will be given
in later sections):
\begin{enumerate}
\item[(1)] the mixed cyclic complex of a Calabi-Yau algebra,
\item[(2)] the mixed cyclic cochain complex of a symmetric Frobenius algebra,
\item[(3)] the mixed Poisson complex of a unimodular Poisson algebra, and
\item[(4)] the mixed Poisson cochain comlex of a unimodular Frobenius Poisson algebra.
\end{enumerate}
As a corollary, the negative cyclic
homology or cyclic cohomology of the above four types of complexes
all have a gravity algebra structure. These structures (the brackets) are nontrivial in general;
for example, it contains a graded Lie algebra which is related to the deformations
of the corresponding algebras, which are in general obstructed (see also Remark \ref{rmk:obstructednessofdeformation} (2)). 
We make the following two remarks:

\begin{remark}(i) The above case (1) gives an answer to Ward's problem at the homology level. 
In loc cit he also
asked for a homotopical, or say, chain level, construction of the Batalin-Vilkovisky and gravity structure 
for general Calabi-Yau algebras. 
Such construction seems to exist; see Remark \ref{fromhomotopyBVtogravity} for some more discussions.


(ii) The above case (2) covers the example
of symmetric Frobenius algebras, which is not new.
In fact, in \cite{Ward2} Ward showed that the mixed cyclic cochain complex of symmetric Frobenius algebras
is an algebra over the (chain) gravity operad, hence Theorem \ref{mainthm} is a direct corollary of his result
in this case. 
In the subsequent paper \cite{CW} 
Campos and Ward proved a more general result saying that
any mixed complex over the (chain) gravity operad induces the gravity algebra
structure on its negative cyclic homology (see loc cit Corollary 1.33 and Example 1.34).
\end{remark}

It will then be very interesting to find the relationships among 
these gravity algebras.
In fact, we find that:
\begin{enumerate}
\item[(a)] if the algebras in the above cases (1) and (2) are Koszul dual to each other, 
then the corresponding two gravity algebras
are isomorphic (partially given in \cite[Theorem 37]{CYZ});
\item[(b)] if the Poisson algebras in the above cases (3) and (4) are
Koszul dual to each other (this means the Poisson algebra
is the polynomials with a quadratic Poisson structure), then the corresponding two
gravity algebras are also isomorphic;
\item[(c)] by a result of Dolgushev \cite{Dolgushev} saying that the deformation quantization of 
a unimodular Poisson algebra is a Calabi-Yau algebra, 
the negative cyclic Poisson homology of the algebra in the above
case (3) is then isomorphic to the 
negative cyclic homology of its deformation
quantization given in the case (1);
\item[(d)] by a result of Felder-Shoikhet \cite{FS} and 
Willwacher-Calaque \cite{WC} saying that the deformation quantization of 
a unimodular Frobenius Poisson algebra is a symmetric Frobenius algebra, 
the cyclic Poisson cohomology of the algebra in the case (4) is then 
isomorphic to the cyclic cohomology of its deformation
quantization given in the case (2).
\end{enumerate}
In particular, starting from quadratic, unimodular Poisson polynomials, 
we have the following correspondence:
\begin{equation}\label{deformationquantizationofKoszulunimodular}
\xymatrixcolsep{6pc}
\xymatrix{
\fbox{\parbox{3.8cm}{
Quadratic unimodular\\
Poisson algebra
}
}\ar@{<->}[r]^{\mbox{\scriptsize{Koszul duality}}}
\ar[d]^{\mbox{\scriptsize{quantization}}}_{\mbox{\scriptsize{deformation}}}
&\fbox{
\parbox{4.4cm}{Quadratic unimodular\\ Frobenius Poisson algebra}}
\ar[d]^{\mbox{\scriptsize{quantization}}}_{\mbox{\scriptsize{deformation}}}
\\
\fbox{\parbox{3.8cm}{Calabi-Yau algebra}}\ar@{<->}[r]^{\mbox{\scriptsize{Koszul duality}}}
&\fbox{\parbox{4.4cm}{Symmetric Frobenius\\ algebra}}
}
\end{equation}
And therefore 
in this case, all the gravity algebras arising from the above (1)--(4) cases
are isomorphic.

This work is a sequel to \cite{CCEY}, where the same classes of algebras
are studied, and the associated Batalin-Vilkovisky algebras are identified. 
We had planned to put the results of this
paper as an appendix of that one, which, however, turns out to be too long.
Moreover, we think Theorem \ref{mainthm} and its corollaries have their own virtue,
and it is better to write a separate paper.
Finally, we mention that
when we were finishing the current paper, we learned of the preprint
of D. Fiorenza and N. Kowalzig \cite{FK}, where the Lie algebra structure on the
negative cyclic homology are also studied.
Nevertheless, our primary goals and main interests are quite different.

\begin{convention}
Throughout this paper,
we work over a ground field $k$ of characteristic $0$ containing $\mathbb R$.
All vector spaces,
their tensors and morphisms 
are graded over $k$ unless otherwise specified.
Algebras are unital and augmented over $k$.
All DG algebras/modules are graded such that the differential has degree $-1$,
and for a chain complex, its cohomology is $\mathrm{H}^\bullet(-):=\mathrm{H}_{-\bullet}(-)$;
for example, an element in the $q$-th Hochschild cochain group $\mathrm{CH}^q(A)$ has grading $-q$. 

\end{convention}

\begin{ack}
We would like to thank Bruno Vallette
and the anonymous referee
for correcting some inaccuracies
and improving the presentations of the paper.
This work is partially supported by NSFC No. 11671281 and 11890663.
\end{ack}

\section{From Batalin-Vilkovisky algebra to gravity algebra}\label{sect:fromBVtogravity}

Gravity algebras are closely related
to Batalin-Vilkovisky algebras.
In this section we first recall the definition of 
Batalin-Vilkovisky and gravity algebras;
some more backgrounds, and especially
their relationships to topological conformal field theories,
can be found in Getzler \cite{Getzler1,Getzler2,Getzler3}.
After that, we prove Theorem \ref{mainthm}.

\begin{definition}[Batalin-Vilkovisky algebra]
Suppose $(V,\hdot)$
is a graded commutative algebra. A {\it Batalin-Vilkovisky
algebra} structure on $V$ is the triple $(V,\hdot, \Delta)$
such that
\begin{enumerate}
\item[$(1)$] $\Delta: V^{i}\to V^{i-1}$ is a differential, that is, $\Delta^2=0$; and
\item[$(2)$] $\Delta$ is a second order operator,
that is,
\begin{eqnarray}\label{Sec_op}
\Delta(a\hdot b\hdot c)
&=& \Delta(a\hdot b)\hdot c+(-1)^{|a|}a\hdot\Delta(b\hdot c)
+(-1)^{(|a|-1)|b|}b\hdot\Delta(a\hdot c)\nonumber \\
& & - (\Delta a)\hdot b\hdot c-(-1)^{|a|}a\hdot(\Delta b)\hdot c
-(-1)^{|a|+|b|}a\hdot b\hdot (\Delta c),
\end{eqnarray}
for all homogeneous $a, b, c\in V$.
\end{enumerate}
\end{definition}

Equivalently,
if we set the bracket
$$
[a,b]:=(-1)^{|a|+1}(\Delta(a\hdot b)-\Delta(a)\hdot b- (-1)^{|a|}a\hdot \Delta(b)),
$$
then $[-,-]$ is a derivation with respect to $\hdot$ for each argument.
In other words, a  Batalin-Vilkovisky algebra is a Gerstenhaber algebra $(V,\hdot, [-,-])$
with a differential $\Delta: V^{i}\to V^{i-1}$ such that
$$
[a,b]=(-1)^{|a|+1}(\Delta(a\hdot b)-\Delta(a)\hdot b-(-1)^{|a|}a\hdot \Delta(b)),
$$
for $a, b\in V$. $\Delta$ is called the {\it Batalin-Vilkovisky operator}, or the
{\it generator} of the Gerstenhaber bracket.

From (\ref{Sec_op}) one also obtains that $\Delta$ being of second order implies that
\begin{multline}
\Delta(x_1\hdot x_2\hdot \cdots\hdot x_n)\\
=\sum_{i<j}(-1)^{\epsilon_{ij}}
 \Delta(x_i\hdot x_j)\hdot x_1\cdots\hat x_i\cdots\hat x_j\cdots x_n
 -\sum_{i=1}^n(-1)^{|x_1|+\cdots+|x_{i-1}|} x_1\cdots\Delta (x_i)\cdots x_n,
\end{multline}
where $\epsilon_{ij}={(|x_1|+\cdots+|x_{i-1}|)|x_i|+(|x_1|+\cdots+|x_{j-1}|)|x_j|-|x_i||x_j|}$.

\begin{definition}[Gravity algebra]\label{def:gravityalg}
Suppose $V$ is a (graded) vector space over $k$.
Then a gravity algebra structure on $V$ consists
of a sequence of (graded) skew symmetric operators (called the {\it higher Lie brackets})
$$\{x_1,\cdots,x_n\}: V^{\otimes n}\to V, \quad n=2,3, \cdots$$
such that
\begin{multline}\label{gen_Jacobi}
\sum_{1\le i<j\le n}(-1)^{\epsilon_{ij}}
\{\{x_i,x_j\},x_1,\cdots, \hat x_i,\cdots,\hat x_j,\cdots,x_n,y_1,\cdots,y_m\}\\
=\left\{
\begin{array}{cl}
\{\{x_1,\cdots,x_n\},y_1,\cdots, y_m\},&\mbox{if}\; m>0,\\
0,&\mbox{otherwise},
\end{array}
\right.
\end{multline}
where $\epsilon_{ij}=(|x_i|+1)(|x_1|+\cdots+|x_{i-1}|+i-1)+(|x_j|+1)(|x_1|+\cdots+|x_{j-1}|+j-1)-(|x_i|+1)(|x_j|+1)$.
\end{definition}

From the definition, one also obtains that $(V, \{-,-\})$ forms a graded Lie algebra.

\begin{definition}[Cyclic homology; {\it cf.} Jones \cite{Jones} and Kassel \cite{Kassel}]\label{def_HC}
Suppose $(\mathrm C_\bullet, b, B)$ is a mixed complex, with $|d|=-1$ and $|B|=1$.
Let
$u$ be a free variable of degree $-2$ which commutes with $b$ and $B$.
The {\it negative cyclic, periodic cyclic, }
and {\it cyclic} chain complex of $\mathrm C_\bullet$ are the following complexes
\begin{eqnarray*}
(\mathrm C_\bullet[\![u]\!], b+uB),&& \\
(\mathrm C_\bullet[\![u, u^{-1}], b+uB),&& \\
(\mathrm C_\bullet[\![u, u^{-1}]/u\mathrm C_\bullet[\![u]\!], b+uB),
\end{eqnarray*}
and are denoted by $\mathrm{CC}^{-}_\bullet(\mathrm C_\bullet) ,\mathrm{CC}^{\mathrm{per}}_\bullet(\mathrm C_\bullet)$
and $\mathrm{CC}_\bullet (\mathrm C_\bullet)$ respectively.
The associated homology are called the {\it negative cyclic, periodic cyclic}
and {\it cyclic homology} of $\mathrm C_\bullet$, and are denoted by
$\mathrm{HC}_\bullet^{-}(\mathrm C_\bullet)$, $\mathrm{HC}_\bullet^{\mathrm{per}}(\mathrm C_\bullet)$ and
$\mathrm{HC}_\bullet(\mathrm C_\bullet)$ respectively.
\end{definition}

In the following, if $\mathrm C_\bullet$ is clear from the context, we sometimes simply
write, for example, $\mathrm{HC}_\bullet^{-}$ instead of $\mathrm{HC}_\bullet^{-}(\mathrm C_\bullet)$.
As usual, $\mathrm{HH}_\bullet$ denotes the Hochschild homology of $\mathrm C_\bullet$,
which is the $b$-homology of the mixed complex.

\begin{remark}[Cyclic cohomology]
Suppose $(\mathrm C^\bullet, b, B)$ is a mixed cochain complex, namely
$|b|=1$ and $|B|=-1$. By negating the degrees of $\mathrm C^\bullet$, we obtain a
mixed chain complex, denoted by $(\mathrm C_{-\bullet}, b, B)$ with $|b|=-1$ and $|B|=1$.
By our convention,
the {\it cyclic cohomology}
of $(\mathrm C^\bullet, b, B)$, denoted by $\mathrm{HC}^\bullet(\mathrm C^\bullet)$, 
is the {\it cohomology}
of the negative cyclic complex of $(\mathrm C_{-\bullet}, b, B)$.
\end{remark}


Consider the short exact sequence
$$0\longrightarrow
u\cdot \mathrm{CC}_{\bullet+2}^{-} 
\stackrel{\iota}\longrightarrow
\mathrm{CC}_\bullet^{-} 
\stackrel{\pi}\longrightarrow
\mathrm{C}_\bullet \longrightarrow 0,$$
where
$\iota: u\cdot \mathrm{CC}_{\bullet+2}^{-} 
\to
\mathrm{CC}_\bullet^{-} 
$ is the embedding
and
$$
\begin{array}{cccl}
\pi:&\mathrm{CC}_\bullet^{-} &
\longrightarrow
&
\mathrm{C}_\bullet \\
&\displaystyle\sum_i x_i\cdot u^i&\longmapsto& x_0
\end{array}
$$
is the projection.
It induces functorially a long exact sequence
\begin{equation}\label{longexseq:Connes}
\cdots\longrightarrow
\mathrm{HC}_{\bullet+2}^{-} 
\longrightarrow
\mathrm{HC}_\bullet^{-} 
\stackrel{\pi_*}\longrightarrow
\mathrm{HH}_\bullet 
\stackrel{\beta}\longrightarrow
\mathrm{HC}_{\bullet+1}^{-} 
\longrightarrow\cdots,
\end{equation}
where we have applied that
$\mathrm{HC}_\bullet^{-}(u \cdot \mathrm{CC}_{\bullet+2}^{-} )\cong\mathrm{HC}_{\bullet+2}^{-} .$
It is obvious that $\beta\circ\pi_*=0$ and we claim that ({\it cf.} \cite{CS,CYZ}):

\begin{lemma}With the above notations, we have
$$\pi_*\circ\beta=B:\mathrm{HH}_\bullet(\mathrm C_\bullet, b)\to\mathrm{HH}_{\bullet+1}(\mathrm C_\bullet, b).$$
\end{lemma}

\begin{proof}
In fact, for any $x\in \mathrm{C}_\bullet $ which is $b$-closed,
from the following diagram
$$
\xymatrix{
&\ar[d]&\ar[d]&\ar[d]&\\
0\ar[r]
&u\cdot \mathrm{CC}_\bullet^{-} \ar[r]^{\iota}\ar[d]^{b+uB}
&\mathrm{CC}_\bullet^{-} \ar[d]^{b+uB}\ar[r]^{\pi}&\mathrm{C}_\bullet
\ar[d]^b\ar[r]&0\\
0\ar[r]
&u\cdot \mathrm{CC}_{\bullet-1}^{-} \ar[r]^{\iota}\ar[d]^{b+uB}
&\mathrm{CC}_{\bullet-1}^{-} \ar[d]^{b+uB}\ar[r]^{\pi}&\mathrm{C}_{\bullet-1}
\ar[d]^b\ar[r]&0\\
&&&&}
$$
we have, up to a boundary, 
$$\iota^{-1}\circ(b+uB)\circ\pi^{-1}(x)=\iota^{-1}\circ(b+uB)(x)=\iota^{-1}(u\cdot B(x))=u\cdot B(x)
\in u\cdot\mathrm{CC}_{\bullet-1}^{-} .$$
Via the isomorphism
$u \cdot \mathrm{CC}_\bullet^{-} \cong\mathrm{CC}_{\bullet+2}^{-} $,
this element $u\cdot B(x)$ is mapped to $B(x)\in\mathrm{CC}_{\bullet+1}^{-} $,
and under $\pi$ it is mapped to $B(x)$.
Thus $\pi_*\circ\beta=B$ as desired.
\end{proof}

\begin{proof}[Proof of Theorem \ref{mainthm}]
(1) Recall that for homogeneous
$x_1,x_2,\cdots,x_n \in\mathrm{HC}_\bullet^{-}$,
\begin{equation}\label{constructionofbrackets}
\{x_1, x_2,\cdots, x_n\}:=(-1)^{(n-1)|x_1|+(n-2)|x_2|+\cdots+|x_{n-1}|}\beta\big(\pi_*(x_1)
\hdot \pi_*(x_2)\hdot\cdots\hdot
\pi_*(x_n)\big).
\end{equation}

We first show the graded skew-symmetry.
Since the multiplication $\hdot$ on $\mathrm{HH}_\bullet$ is graded commutative,
by commutativity it is sufficient to show
$$\{x_1, x_2,\cdots, x_n\}+(-1)^{(|x_1|+1)(|x_2|+1)}\{x_2, x_1,\cdots, x_n\}=0.$$
In fact,
\begin{eqnarray*}
&& \{x_1, x_2,\cdots, x_n\}+(-1)^{(|x_1|+1)(|x_2|+1)}\{x_2, x_1,\cdots, x_n\}\\
&=&(-1)^{(n-1)|x_1|+(n-2)|x_2|+\cdots+|x_{n-1}|}\beta\big(\pi_*(x_1)\hdot \pi_*(x_2)\hdot\cdots\hdot
\pi_*(x_n)\big)\\
&&+(-1)^{(|x_1|+1)(|x_2|+1)}\cdot(-1)^{(n-1)|x_2|+(n-2)|x_1|+\cdots+|x_{n-1}|}\beta\big(\pi_*(x_2)\hdot \pi_*(x_1)\hdot\cdots\hdot
\pi_*(x_n)\big)\\
&=&(-1)^{(n-1)|x_1|+(n-2)|x_2|+\cdots+|x_{n-1}|}\beta\big(\pi_*(x_1)\hdot \pi_*(x_2)\hdot\cdots\hdot
\pi_*(x_n)\big)\\
&&-(-1)^{|x_1||x_2|+(n-2)|x_2|+(n-1)|x_1|+\cdots+|x_{n-1}|}\cdot(-1)^{|x_1||x_2|}
\beta\big(\pi_*(x_1)\hdot \pi_*(x_2)\hdot\cdots\hdot
\pi_*(x_n)\big)\\
&=&0.
\end{eqnarray*}

Next, we show the graded Jacobi identity.
First,
\begin{eqnarray*}
&&\sum_{i<j}(-1)^{\epsilon_{ij}}\{\{x_i,x_j\},x_1,\cdots,\hat x_i,\cdots,\hat x_j,\cdots,x_n\}\\
&=&\sum_{i<j}(-1)^{\epsilon_{ij}}\beta\big(\pi_*(\beta(\pi_*(x_i)\hdot\pi_*(x_j)))\hdot\pi_*(x_1)\cdots
\widehat{\pi_*(x_i)}\cdots \widehat{\pi_*(x_j)}\cdots\pi_*(x_n)\big)\\
&=&\sum_{i<j}(-1)^{\epsilon_{ij}}\beta\big(B(\pi_*(x_i)\hdot\pi_*(x_j))\hdot\pi_*(x_1)\cdots
\widehat{\pi_*(x_i)}\cdots \widehat{\pi_*(x_j)}\cdots\pi_*(x_n)\big)\\
&=&\beta\big(B(\pi_*(x_1)\cdots\pi_*(x_n))+\sum_i(-1)^{|x_1|+\cdots+|x_{i-1}|}\pi_*(x_1)\cdots B(\pi_*(x_i))\cdots\pi_*(x_n)\big)\\
&=&0,
\end{eqnarray*}
where
$\epsilon_{ij}=|x_i|(|x_1|+\cdots+|x_{i-1}|)+|x_j|(|x_1|+\cdots+|x_{j-1}|)-|x_i||x_j|$,
and
in the last equality, we have used the fact that $\beta\circ B=\beta\circ\pi_*\circ \beta=0$
and $B\circ\pi_*=\pi_*\circ\beta\circ\pi_*=0$.
Second, by the same argument,
\begin{eqnarray*}
&&\sum_{i<j}(-1)^{\epsilon_{ij}}\{\{x_i,x_j\},x_1,\cdots,\hat x_i,\cdots,\hat x_j,\cdots,x_n,y_1,\cdots,y_\ell\}\\
&=&
\beta\Big((-1)^{\epsilon_{ij}}\big(B(\pi_*(x_1)\cdots\pi_*(x_n))\\
&&+\sum_i(-1)^{|x_1|+\cdots+|x_{i-1}|}\pi_*(x_1)\cdots B(\pi_*(x_i))\cdots\pi_*(x_n)\big)
\pi_*(y_1)\cdots\pi_*(y_\ell)\Big),
\end{eqnarray*}
which is then equal to
\begin{eqnarray*}
&&
\beta\Big(\big(B(\pi_*(x_1)\cdots\pi_*(x_n))\big)
\pi_*(y_1)\cdots\pi_*(y_\ell)\Big)\\
&=&
\beta\Big(\big(\pi_*\circ\beta(\pi_*(x_1)\cdots\pi_*(x_n))\big)
\pi_*(y_1)\cdots\pi_*(y_\ell)\Big)\\
&=&\{\{x_1,\cdots,x_n\},y_1,\cdots,y_\ell\}.
\end{eqnarray*}
This proves the statement.

(2) Since $f:(\mathrm C_\bullet, b, B)\to (\mathrm C'_\bullet, b', B')$
is a quasi-isomorphism of mixed complexes, we have the following commutative diagram
of long exact sequences
$$
\xymatrix{
\cdots\ar[r]&
\mathrm{HC}_{\bullet+2}^{-}(\mathrm C)\ar[d]_{f_*}^{\cong}
\ar[r]&
\mathrm{HC}_\bullet^{-} (\mathrm C)\ar[d]_{f_*}^{\cong}\ar[r]^-{\pi_*}
&
\mathrm{HH}_\bullet(\mathrm C)\ar[d]_{f_*}^{\cong}\ar[r]^-{\beta}
&
\mathrm{HC}_{\bullet+1}^{-}(\mathrm C)\ar[d]_{f_*}^{\cong}\ar[r]&\cdots\\
\cdots\ar[r]&
\mathrm{HC}_{\bullet+2}^{-}(\mathrm C')
\ar[r]&
\mathrm{HC}_\bullet^{-}(\mathrm C')
\ar[r]^-{\pi_*}
&
\mathrm{HH}_\bullet(\mathrm C')
\ar[r]^-{\beta}
&
\mathrm{HC}_{\bullet+1}^{-}(\mathrm C')
\ar[r]&\cdots
}
$$
The statement now follows directly
from the construction \eqref{constructionofbrackets} of the
higher brackets.
\end{proof}

\begin{remark}\label{fromhomotopyBVtogravity}
In the above proof, 
in particular in \eqref{constructionofbrackets}, we
have used the Batalin-Vilkovisky algebra structure
on the Hochschild homology. This assumption
seems to be too strong; 
in fact, a chain level (up to homotopy) Batalin-Vilkovisky structure is enough.
This gives a hint to completely solve Ward's question for general
Calabi-Yau algebras on the homotopy level.
The technical difficulty for us is that at present, we are not able to show
the chain level noncommutative Poincar\'e duality and hence the
homotopy Batalin-Vilkovisky structure for general Calabi-Yau algebras.
We hope to address this problem in the future.
\end{remark}

\section{Differential calculus with duality and Batalin-Vilkovisky algebras}\label{Sect_BV}

In this section, we present how to obtain the Batalin-Vilkovisky algebra structure 
from a special class of mixed complexes.
Such mixed complexes arise from the so-called {\it differential calculus with duality},
a notion introduced by Lambre in \cite{Lambre}.

Lambre's notion is based on 
Tamarkin-Tsygan's notion of
{\it differential calculus}
(see \cite{TT05}), which
is to capture the algebraic structure that occurs
on the Hochschild cohomology and homology
of associative algebras.
He observed that, for algebras such as Calabi-Yau algebras
and symmetric Frobenius algebras, the {\it volume form}
which gives the so-called noncommutative Poincar\'e duality
(studied by Van den Bergh \cite{VdB0}),
can be capsuled into the differential calculus and
forms what he called differential calculus with duality.


\begin{definition}[Tamarkin-Tsygan \cite{TT05}]
Let $\mathrm{H}^{\bullet}$ and $\mathrm{H}_{\bullet}$ be graded vector spaces.
A \textit{differential calculus} is the sextuple
$$
(\mathrm{H}^{\bullet},\mathrm{H}_{\bullet}, \cup, \cap, \{-,-\}, B),
$$
satisfying the following conditions
\begin{enumerate}
  \item $(\mathrm{H}^{\bullet},\cup, \{-,-\})$ is a Gerstenhaber algebra;
  \item $\mathrm{H}_{\bullet}$ is a graded (left) module over $(\mathrm{H}^{\bullet}, \cup)$ by the map
        $$
        \cap: \mathrm{H}^{n}\otimes\mathrm{H}_{m}\to \mathrm{H}_{m-n},\; f\otimes \alpha \mapsto f\cap \alpha,
        $$
        for any $f\in \mathrm{H}^{n}$ and $\alpha \in \mathrm{H}_{m}$,
        i.e., if we define $\iota_{f} (\alpha):= f \cap\alpha$, then $\iota_{f}\iota_{g}=\iota_{f\cup g}$;
  \item There is a map $B: \mathrm{H}_{\bullet}\to \mathrm{H}_{\bullet+1}$ satisfying 
  $B^2=0$ and moreover, if we set $L_f:=[B,\iota_f]=B\iota_f-(-1)^{|f|}\iota_f B$, then
  $$[L_f, L_g]=L_{\{f,g\}}$$
  and
        \begin{equation}\label{compatibilityofGerstenmodule}
        (-1)^{|f|+1}\iota_{\{f,g\}}=[L_f,\iota_g]:=L_{f}\iota_g-(-1)^{|g|(|f|+1)}\iota_g L_f.
        \end{equation}
\end{enumerate}
\end{definition}

\begin{remark}[Comparison with Poisson modules]\label{rem:comparisonofPoisson}
Recall that if $(R, \bullet, \{-,-\})$ is a Poisson algebra. 
A Poisson $R$-module,
is a $k$-vector space, say $M$, equipped with the structures of
an algebra module (denoted by $\circ$) 
and a Lie module (denoted by $\{-,-\}_M$) over $R$, together with the following
compatibility condition
\begin{equation}\label{compatibilityofPoissonmodule}
\{r, s\circ m\}_M=
\{r, s\}\circ m+s\circ\{r,m\}_M,
\end{equation}
for all $r,s\in R$ and $m\in M$.
We may similarly consider modules over graded Poisson algebras,
or even Gerstenhaber algebras. 
Now for a differential calculus $
(\mathrm{H}^{\bullet},\mathrm{H}_{\bullet}, \cup, \cap, \{-,-\}, B),
$ given in above definition,
(2) says that $\mathrm H_\bullet$ is an algebra module over $\mathrm H^\bullet$,
and (3) says that $\mathrm H_\bullet$ is a Lie module over $\mathrm H^\bullet$
and that
these two module structures are compatible by 
\eqref{compatibilityofGerstenmodule}, which is exactly parallel to
\eqref{compatibilityofPoissonmodule}.
In other words, for a differential calculus as above,
$\mathrm H_\bullet$ is a Gerstenhaber module over $\mathrm H^\bullet$.
\end{remark}

\begin{definition}[Lambre \cite{Lambre}]
A differential calculus $(\mathrm{H}^{\bullet},\mathrm{H}_{\bullet}, \cup, \cap, \{-,-\}, B)$ is
called a \textit{differential calculus with duality} if there exists an integer $n$ and an element $\eta\in \mathrm{H}_{n}$
such that $\eta \cap 1=\eta$ and $B(\eta)=0$,
and for any $i\in \mathbb Z$,
$$
\mathrm{PD}(-):=\eta\cap- : \mathrm{H}^{i}\to \mathrm{H}_{n-i}
$$
is an isomorphism. Such isomorphism is called the \textit{noncommutative Poincar\'e duality},
and $\eta$ is called the {\it volume form}.
\end{definition}


Suppose
$
(\mathrm{H}^{\bullet},\mathrm{H}_{\bullet}, \cup, \cap, \{-,-\}, B),
$
is a differential calculus with duality.
Let $\Delta: \mathrm{H}^\bullet\to\mathrm{H}^{\bullet-1}$ be the pull-back of $B$ via the map PD:
$$
\xymatrixcolsep{3.5pc}
\xymatrix{
  \mathrm{H}^{\bullet} \ar[d]_{\mathrm{PD}} \ar[r]^{\Delta} & \mathrm{H}^{\bullet-1}  \ar[d]^{\mathrm{PD}} \\
  \mathrm{H}_{n-\bullet} \ar[r]^{B} & \mathrm{H}_{n-\bullet+1}. }
$$
In other words, $\Delta=\mathrm{PD}^{-1}\circ B\circ \mathrm{PD}$, which is the divergence
of $B$.

\begin{theorem}[Lambre \cite{Lambre}]\label{Thm_Lambre}
Suppose that $(\mathrm{H}^{\bullet},\mathrm{H}_{\bullet},\cup,\cap,\{-,-\},B,\eta)$
is a differential calculus with duality.
Then the triple $(\mathrm{H}^{\bullet},\cup,\Delta)$ is a Batalin-Vilkovisky algebra.
\end{theorem}

\begin{proof}See Lambre \cite{Lambre} (see also \cite[Theorem 5.3]{CCEY} for some more details).
\end{proof}

Alternatively, we may push the cup product on $\mathrm H^\bullet$ via $\mathrm{PD}$
to $\mathrm H_\bullet$, which is again denoted by $\cap$, then
$(\mathrm H_\bullet, \cap, B)$, after degree shifting down by $n$,
forms a Batalin-Vilkovisky algebra.
In practice, $\mathrm H_\bullet$ is usually the homology of 
a mixed complex. In this case, we have the following.

\begin{theorem}\label{BVcriterion}
Suppose $(\mathrm{H}^{\bullet},\mathrm{H}_{\bullet},\cup,\cap,\{-,-\},B,\eta)$
is a differential calculus with duality, where
$\mathrm H_\bullet=\mathrm H_\bullet(\mathrm C_\bullet, b)$ for some mixed
complex $(\mathrm C_\bullet, b, B)$, then the negative cyclic homology
$\mathrm{HC}_\bullet^{-}(\mathrm C_\bullet)$, after degree shifting down by $n-2$,
is a gravity algebra.
\end{theorem}

\begin{proof}
This is a direct corollary of Theorems \ref{mainthm} and \ref{Thm_Lambre}.
\end{proof}

\section{Calabi-Yau and symmetric Frobenius algebras}

In this section, we show that the Hochschild cohomology
and homology of Calabi-Yau algebras (respectively symmetric Frobenius
algebras) satisfy the condition of Theorem \ref{BVcriterion}, and thus their
negative cyclic homology (respectively cyclic cohomology) has a gravity algebra structure.
This completes the cases (1) and (2) in \S\ref{Sect_intro}.

We refer the reader to Loday \cite{Loday} for the definition of Hochschild homology and cohomology.
Suppose $A$ is an associative algebra over $k$.
Let $\bar A:=A/k$ be its augmentation. Let $A\to\bar A: a\mapsto \bar a$
be the projection.
Denote by $(\bar{\mathrm{C}}^\bullet(A),\delta)$ the reduced
Hochschild cochain complex of $A$. Let us recall that:

\begin{enumerate}
\item[$(1)$]
The {\it Gerstenhaber cup product} on
$\bar{\mathrm{C}}^\bullet(A)$
is given as follows: for any 
$f\in \bar{\mathrm{C}}^{n}(A)$ and $g \in \bar{\mathrm{C}}^{m}(A)$,
$$
f\cup g(\bar{a}_1,\ldots, \bar{a}_{n+m})
:=(-1)^{nm}
f(\bar{a}_1, \ldots, \bar{a}_{n})g(\bar{a}_{n+1},\ldots,\bar{a}_{n+m}),
$$
where $(\bar{a}_{n+1},\ldots,\bar{a}_{n+m})\in \bar{A}^{\otimes (n+m)}$.

\item[$(2)$]
For any $f\in \bar{\mathrm{C}}^{n}(A)$ and $g \in \bar{\mathrm{C}}^{m}(A)$,
let
\begin{multline*}
f {\circ} g (\bar{a}_1,\ldots, \bar{a}_{n+m-1})\\
:=
\sum^{n-1}_{i=0}(-1)^{(|g|+1)i} f(\bar{a}_1, \ldots, \bar{a}_{i},
\overline{g(\bar{a}_{i+1},\ldots, \bar{a}_{i+m})}, \bar{a}_{i+m+1},\ldots, \bar{a}_{n+m-1}),
\end{multline*}
then the {\it Gerstenhaber bracket} on $\bar{\mathrm{C}}^\bullet(A)$
is defined to be
$$
\{f,g\}:=f\circ g-(-1)^{(|f|+1)(|g|+1)}g\circ f.
$$
\end{enumerate}

\begin{theorem}[Gerstenhaber \cite{Gerstenhaber}]\label{G-algebra}
Let $A$ be an associative algebra.
Then the Hochschild cohomology $\mathrm{HH}^{\bullet}(A)$ of $A$ equipped with the
Gerstenhaber cup product and the Gerstenhaber bracket forms a Gerstenhaber algebra.
\end{theorem}



Now let $(\bar{\mathrm C}_\bullet(A),\partial)$ be the reduced
Hochschild chain complex of $A$. 
Then
$\bar{\mathrm{C}}^\bullet(A)$ acts on
$\bar{\mathrm{C}}_\bullet(A)$, called the {\it cap product},
as follows:
for homogeneous
$f\in \bar{\mathrm{C}}^{n}(A)$ and $\alpha=(a_0,\bar{a}_1, \ldots, \bar{a}_m)\in \bar{\mathrm{C}}_{m}(A)$,
the \textit{cap product}
$
\cap: \bar{\mathrm{C}}^{n}(A)\times \bar{\mathrm{C}}_{m}(A)\to \bar{\mathrm{C}}_{m-n}(A)
$
is given by
$$
f \cap \alpha  :=(a_0f(\bar{a}_1,\ldots,\bar{a}_n),\bar{a}_{n+1},\ldots,\bar{a}_m),
$$ for
$m\geq n$,
and zero otherwise.
We denote by $\iota_{f}(-):= f\cap -$ the contraction operator, 
then $\iota_{f}\iota_{g}=(-1)^{|f||g|}\iota_{g\cup f}$.

Recall the \textit{Connes operator}
$
B: \bar{\mathrm{C}}_{\bullet}(A)\to \bar{\mathrm{C}}_{\bullet+1} (A)
$
is given by
$$
B(\alpha)
:=\sum_{i=0}^{m} (-1)^{mi} (1,  \bar{a}_i, \cdots,   \bar{a}_m, \bar{a}_0,\cdots, \bar{a}_{i-1}).
$$
It is known that $(\bar{\mathrm C}_\bullet(A), \partial, B)$ forms a mixed complex. 

From $\iota$ and $B$, we may define the Lie derivative of 
$\bar{\mathrm C}^\bullet(A)$ on
$\bar{\mathrm C}_\bullet(A)$ by $L:=[B, \iota]$.
It is then direct to see
$\partial=L_{\mu}$, where $\mu$ is the product of $A$, viewed as a 
cochain in $\bar{\mathrm C}^{-2}(A)$.

\begin{theorem}[Daletskii-Gelfand-Tsygan \cite{DGT}; Tamarkin-Tsygan \cite{TT05}]
\label{DGT}
Let $A$ be an associative algebra.
Denote by $\mathrm{HH}^\bullet(A)$
and $\mathrm{HH}_\bullet(A)$ the Hochschild
cohomology and homology of $A$ respectively. Then the following sextuple
$$
(\mathrm{HH}^{\bullet}(A),\mathrm{HH}_{\bullet}(A), \cup, \cap, \{-,-\}, B)
$$
is a differential calculus.
\end{theorem}

We remark that on the chain level, some identities in the above theorem
only hold up to homotopy, which is highly nontrivial.

In the rest of this section, 
we give two examples of differential calculus with duality,
one from Calabi-Yau algebras and the other from symmetric Frobenius algebras.
They are obtained by Lambre in \cite{Lambre}, which, in particular, put the
Batalin-Vilkovisky algebra structure, first obtained by Ginzburg 
\cite[Theorem 3.4.3]{Ginzburg} and Tradler
\cite[Theorem 1]{Tradler}, in a general framework.
A new observation here, if there is any, is that these two examples
satisfy the condition of Theorem \ref{BVcriterion}, and hence we 
obtain a gravity algebra structure on the associated negative cyclic homology.

\subsection{Calabi-Yau algebras}\label{subsect:CY}

The notion of Calabi-Yau algebra is introduced by Ginzburg in \cite{Ginzburg}.
In this subsection, we briefly 
recall the differential calculus structure on Calabi-Yau algebras.

\begin{definition}[Calabi-Yau algebra]
An algebra $A$ is called a \textit{Calabi-Yau algebra of dimension $d$} (or $d$-Calabi-Yau algebra)
if
\begin{enumerate}
\item[$(1)$] $A$ is homologically smooth,
that is, $A$, viewed as an $A^e$-module, has a bounded resolution
by finitely generated projective $A^e$-modules,
and
\item[$(2)$] there is an isomorphism
\begin{equation}\label{CY_cond}
\eta:\mathrm{RHom}_{A^e}(A, A\otimes A)\rightarrow \Sigma^{-d}A
\end{equation}
in the derived category $D(A^e)$ of (left) $A^e$-modules,
where $A^e$ is the enveloping algebra
of $A$ and $\Sigma^{-d}(-)$ is the $d$-fold desuspension functor.
\end{enumerate}
\end{definition}

\begin{theorem}[de Thanhoffer de V\"olcsey-Van den Bergh \cite{dTdVVdB}, 
Lambre \cite{Lambre}]
\label{thm:differentialcalculuswithdualityforCY}
If $A$ is $n$-Calabi-Yau, then there
exists a volume form $\Omega\in\mathrm{HH}_n(A)$ such that
$$
\begin{array}{cccl}
\mathrm{PD:}&\mathrm{HH}^i(A)&\longrightarrow&\mathrm{HH}_{n-i}(A)\\
&f&\longmapsto&f\cap\Omega
\end{array}
$$
is an isomorphism. That is,
$$
(\mathrm{HH}^{\bullet}(A),\mathrm{HH}_{\bullet}(A), \cup, \cap, \{-,-\}, B)
$$
is in fact a differential calculus with duality.
\end{theorem}

For a proof of this statement, see \cite[Proposition 5.5]{dTdVVdB} and also \cite{Lambre}.
Note that $\mathrm{HH}_\bullet(A)$ is the $b$-homology of the mixed complex
$(\bar{\mathrm C}_\bullet(A), b, B)$, which satisfies the condition of Theorem
\ref{BVcriterion}, and hence we obtain
a gravity algebra structure on its negative cyclic homology, which may be viewed
as induced from the Batalin-Vilkovisky structure on the Hochschild cohomology.

\subsection{Symmetric Frobenius algebras}

For an associative algebra $A$,
denote $A^*:=\mathrm{Hom}(A, k)$, and then $A^*$ is an $A$-bimodule.
Let $\bar{\mathrm C}^\bullet(A; A^*)$ be the reduced Hochschild cochain complex
of $A$ with values in $A^*$.
Under the identities 
\begin{equation}\label{idofHochschildcochainwithchain}
\bar{\mathrm{C}}^\bullet(A; A^*)=\bigoplus_{n\geq 0}\mathrm{Hom}(\bar{A}^{\otimes n}, A^*)
=\bigoplus_{n\geq 0}\mathrm{Hom}(A\otimes \bar{A}^{\otimes n}, k)=
\mathrm{Hom}(\bar{\mathrm C}_\bullet(A), k),
\end{equation}
it is proved, for example, in \cite[\S2.4]{Loday}, that
the Hochschild coboundary of the leftmost is identical with
the dual of the Hochschild boundary of the rightmost.
Moreover, via the above identity,
one may equip on $\bar{\mathrm{C}}^\bullet(A; A^*)$
the dual Connes differential, which is denoted by $B^*$, 
i.e., $B^{*}(g):=(-1)^{|g|}g\circ B$ for homogeneous $g \in \bar{\mathrm{C}}^\bullet(A; A^*)$.
It is then direct to see that $B^*$ has square zero and 
commutes with the Hochschild coboundary operator.

The Hochschild cochain complex $\bar{\mathrm C}^\bullet(A)$
acts on $\bar{\mathrm C}^\bullet(A; A^*)$ via the following ``cap product" 
$$
\begin{array}{cccl}
\cap^*:&\bar{\mathrm{C}}^{\bullet}(A)\times \bar{\mathrm{C}}^{\bullet}(A; A^*)
&\stackrel{\cap^{\ast}}{\longrightarrow}&\bar{\mathrm{C}}^\bullet(A; A^*)\\
&(f, \quad\quad \alpha)&\longmapsto&
(-1)^{|f||\alpha|}\alpha\circ \iota_f,
\end{array}
$$
which respects the Hochschild coboundaries. We thus obtain the following.

\begin{theorem}
\label{dual-Hoch-DC}
Let $A$ be an associative algebra.
Then
$$
(\mathrm{HH}^\bullet(A),\mathrm{HH}^\bullet(A; A^*),\cup,\cap^{\ast},\{-,-\}, B^*)
$$
is a differential calculus.
\end{theorem}

\begin{proof}
We already know from Theorem \ref{DGT} that
$$
(\mathrm{HH}^{\bullet}(A),\mathrm{HH}_{\bullet}(A), \cup, \cap, \{-,-\}, B)
$$
is a differential calculus.
Note that by \eqref{idofHochschildcochainwithchain},
$\mathrm{HH}^\bullet(A; A^*)$
is the linear dual of $\mathrm{HH}_\bullet(A)$, and therefore
by Remark \ref{rem:comparisonofPoisson},
$\mathrm{HH}^\bullet(A; A^*)$ is a Gerstenhaber module over
$\mathrm{HH}^\bullet(A)$ given by the adjoint action.
Since $L^*=[B, \cap]^*=[B^*,\cap^*]$, 
the differential calculus structure is obtained as desired.
\end{proof}

The above theorem can be directly generalized to graded algebras without any difficulty.
Now let us recall that a graded associative algebra $A$ is called {\it symmetric Frobenius} 
(or {\it cyclic associative}) 
of degree $n$
if there is a non-degenerate bilinear pairing
$$
\langle-,-\rangle: A\times A\to k
$$
which is cyclically invariant; that is,
$\langle a\hdot b, c\rangle=(-1)^{|c|(|a|+|b|)}\langle c\hdot a,b\rangle$,
for all homogeneous $a,b,c\in A$.
It is direct to see that this is equivalent to a degree $n$
isomorphism of $A$-bimodules
$$
\eta: A\to A^*,
$$
which is completely determined by the image of $1$.
Viewing $\eta\in\bar{\mathrm{C}}^{n}(A; A^*)$,
then that $\eta$ is an $A$-bimodule 
map means $\delta(\eta)=0$ and hence $\eta$
is a Hochschild cocycle. The isomorphism
of $\eta$ in fact means $[\eta]\ne 0$.
Let
$$
\mathrm{PD}:\mathrm{HH}^\bullet(A)\to\mathrm{HH}^{\bullet-n}(A; A^*)
$$
be the composition
$$
\mathrm{Hom}(\bar A^{\otimes q}, A)
\stackrel{\eta\circ-}{\longrightarrow}
\mathrm{Hom}(\bar A^{\otimes q}, A^*),
$$
which passes to the cohomology level,
then one obtains the following.

\begin{theorem}[Lambre \cite{Lambre}]\label{thm:differentialcalculuswithdualityforFrob}
Suppose $A$ is a symmetric Frobenius algebra.
Then $$(\mathrm{HH}^\bullet(A), \mathrm{HH}^\bullet(A; A^*))$$ forms a 
differential calculus with duality.
\end{theorem}

Again, $\mathrm{HH}^\bullet(A; A^*)$ is the $\delta$-cohomology of
the mixed complex $(\bar{\mathrm C}^\bullet(A; A^*), \delta, B^*)$,
and by Theorem \ref{BVcriterion}, we obtain a gravity algebra structure 
on its cyclic cohomology,
induced from the Batalin-Vilkovisky structure on the Hochschild cohomology.

\begin{remark}\label{rmk:obstructednessofdeformation}
(1) The Batalin-Vilkovisky algebra structure on the Hochschild
cohomology has been obtained by Tradler \cite{Tradler} (see also Menichi \cite{Menichi}
and Ward \cite{Ward2}).

(2) The two graded Lie algebras that are contained in the gravity algebras after Theorems 
\ref{thm:differentialcalculuswithdualityforCY} and \ref{thm:differentialcalculuswithdualityforFrob}
are nontrivial in general. 
In fact, it is proved by
de Thanhoffer de V\"olcsey and Van den Bergh \cite{dTdVVdB} 
and Terilla and Tradler in \cite{TT} that
the DG Lie algebras that control the deformations of the 
Calabi-Yau algebras and symmetric Frobenius algebras,
are quasi-isomorphic to the negative cyclic complex and cyclic cochain complex
of these two algebras respectively;
the Lie brackets on the homology level are exactly identical to the Lie brackets
contained in the gravity algebra structure.
\end{remark}

\section{Unimodular Poisson and Frobenius Poisson algebras}

In this section, we show that the Poisson chain and cochain
complexes of unimodular Poisson and unimodular Frobenius Poisson
algebras satisfy the condition of Theorem \ref{BVcriterion}.
The differential calculus with duality structure that appears in these
two cases are already implicit in Xu \cite{Xu} and Zhu, Van Oystaeyen and Zhang
\cite{ZVOZ}, where
the Batalin-Vilkovisky algebra are also shown.
Therefore the associated
negative cyclic Poisson homology (respectively cyclic Poisson cohomology) 
has a gravity algebra structure.
This completes the cases (3) and (4) in \S\ref{Sect_intro}.

Suppose $A$ is a (possibly graded)
commutative algebra, and $M$ an $A$-bimodule.
Let $\Omega^p(A)$ be the set of $p$-th K\"ahler differential forms of $A$,
and $\mathfrak X^{p}(A; M)$ be the set of skew-symmetric multilinear maps
$A^{\otimes p}\to M$ which are derivations in each argument. 
There is an isomorphism of left $A$-modules
$$
\mathfrak X^{p}(A; M)\cong\mathrm{Hom}_A(\Omega^p(A), M).
$$
If $M=A$, then we simply write
$\mathfrak X^{p}(A; M)$
as
$\mathfrak X^{p}(A)$.

\begin{definition}[Poisson homology; Koszul \cite{Koszul}]
Suppose $A$ is a Poisson algebra with the Poisson structure $\pi$.
Denote by $\Omega^p(A)$ the set of $p$-th K\"ahler differential forms of $A$.
Then the {\it Poisson chain complex} of $A$, denoted by $\mathrm{CP}_\bullet(A)$,
is
\begin{equation}
\xymatrix{
\cdots\ar[r]&\Omega^{p+1}(A)\ar[r]^{\partial}&\Omega^{p}(A)\ar[r]^{\partial}
&\Omega^{p-1}(A)\ar[r]^{\partial}&\cdots\ar[r]&\Omega^0(A)=A,
}
\end{equation}
where $\partial$ is given by
\begin{eqnarray*}
\partial(f_0 df_1\wedge \cdots\wedge df_p)&=&\sum_{i=1}^p(-1)^{i-1}\{f_0, f_i\} df_1\wedge\cdots \widehat{df_i}\cdots \wedge df_p\\
&+&\sum_{1\le i<j\le p}(-1)^{j-i}f_0 d\{f_i,f_j\}\wedge df_1\wedge\cdots \widehat{df_i}\cdots\widehat{df_j}\cdots \wedge df_p.
\end{eqnarray*}
The associated homology is called the {\it Poisson homology} of $A$, and is denoted by
$\mathrm{HP}_\bullet(A)$.
\end{definition}

In the above definition, $\{f_i, f_j\}$ means $\pi(f_i, f_j)$.
It is direct to check (see also Xu \cite{Xu}) that $\partial$ commutes with
the de Rham differential $d$, and therefore
$(\Omega^\bullet(A), \partial, d)$
is a mixed complex.

\begin{definition}[Poisson cohomology; Lichnerowicz \cite{Lichnerowicz}]
Suppose $A$
is a Poisson algebra and $M$ is a left Poisson $A$-module.
Let
$
\mathfrak X^{p}(A; M)
$
be the space of skew-symmetric multilinear maps
$A^{\otimes p}\to M$ that are derivations in each argument.
The {\it Poisson cochain complex} of $A$ with values in $M$, denoted by $\mathrm{CP}^\bullet(A; M)$,
is the cochain complex
$$
\xymatrix{
M=\mathfrak X^0(A; M)\ar[r]^-{\delta}&\cdots\ar[r]&\mathfrak X^{p-1}(A; M)
\ar[r]^-{\delta}&
\mathfrak X^{p}(A; M)\ar[r]^-{\delta}&\cdots
}$$
where
$\delta$ is given by
\begin{eqnarray*}
\delta(P)(f_0, f_1,\cdots, f_p)&:=&\sum_{0\le i\le p}(-1)^i\{f_i, P(f_0,\cdots, \widehat{f_i},\cdots, f_p)\}\\
&+&\sum_{0\le i<j\le p}(-1)^{i+j}P(\{f_i, f_j\}, f_1, \cdots, \widehat{f_i},\cdots, \widehat{f_j},\cdots, f_p).
\end{eqnarray*}
The associated cohomology is called the {\it Poisson cohomology} of $A$ with values in $M$, and is denoted
by $\mathrm{HP}^\bullet(A; M)$.
In particular, if $M=A$, then the cohomology is just called the {\it Poisson cohomology} of $A$, and is simply
denoted by $\mathrm{HP}^\bullet(A)$.
\end{definition}

Note that in the above definition, the Poisson cochain complex, viewed as a chain complex,
is negatively graded,
and the coboundary $\delta$ has degree $-1$. However, by our convention,
the Poisson cohomology is positively graded.
In the following, we present two versions of differential calculus 
on the Poisson (co)homology.

\subsection{Unimodular Poisson algebras}\label{subsect:unimodularPoisson}

Given a commutative algebra $A$, 
we have the following operations on $\mathfrak{X}^\bullet(A)$ and $\Omega^\bullet(A)$:
\begin{enumerate}
\item[(1)] Wedge (cup) product:
suppose $P\in\mathfrak X^{p}(A)$ and $Q\in\mathfrak X^{q}(A)$,
then the {\it wedge product} of $P$ and $Q$, denoted by $P\wedge Q$, is a polyvector
in $\mathfrak X^{p+q}(A)$, defined by
$$
(P\wedge Q)(f_1,f_2,\cdots, f_{p+q}):=
\sum_{\sigma\in S_{p,q}}\mathrm{sgn}(\sigma)P(f_{\sigma(1)},\cdots, f_{\sigma(p)})
\cdot
Q(f_{\sigma(p+1)},\cdots, f_{\sigma(p+q)}),
$$
where $\sigma$ runs over all $(p,q)$-shuffles of $(1,2,\cdots, p+q)$.

\item[(2)] Schouten bracket:
suppose $P\in\mathfrak X^{p}(A)$ and $Q\in\mathfrak X^{q}(A)$,
then their {\it Schouten bracket}
is an element in $\mathfrak{X}^{p+q-1}(A)$, which is denoted by
$[P,Q]$ and is given by
\begin{multline*}
[P, Q](f_1,f_2,\cdots, f_{p+q-1}):=\sum_{\sigma\in S_{q,p-1}}\mathrm{sgn}(\sigma)
P\big(Q(f_{\sigma(1)},\cdots, f_{\sigma(q)}),f_{\sigma(q+1)},\cdots, f_{\sigma(q+p-1)}\big)\\
-(-1)^{(p-1)(q-1)}\sum_{\sigma\in S_{p,q-1}}\mathrm{sgn}(\sigma)
Q\big(P(f_{\sigma(1)},\cdots, f_{\sigma(p)}),f_{\sigma(p+1)},\cdots,f_{\sigma(p+q-1)}\big).
\end{multline*}

\item[(3)] Contraction (inner product): suppose $P\in\mathfrak X^{p}(A)$ and $\omega=df_1\wedge \cdots\wedge df_n\in\Omega^n(A)$,
then
the {\it contraction} (also called {\it inner product} or {\it internal product}) of $P$ with $\omega$, denoted by
$\iota_P(\omega)$, is an $A$-linear map
with values in $\Omega^{n-p}(A)$
given by
\begin{equation}\label{def:innerproductofvectorfields}
\iota_P(\omega)=\left\{
\begin{array}{cl}\displaystyle
\sum_{\sigma\in S_{p,n-p}}\mathrm{sgn}(\sigma)
P(f_{\sigma(1)},\cdots, f_{\sigma(p)})df_{\sigma(p+1)}\wedge\cdots \wedge df_{\sigma(n)},&\mbox{if}\; n\ge p,\\
0,&\mbox{otherwise.}
\end{array}
\right.
\end{equation}
\end{enumerate}


\begin{proposition}\label{DC-on-Poisson}
Suppose $A$ is a Poisson algebra. Then 
$$
(\mathrm{HP}^\bullet(A), \mathrm{HP}_\bullet(A), \wedge,\iota,[-,-], d)
$$
forms a differential calculus, where $d$ is the de Rham differential.
\end{proposition}
\begin{proof}
We only need to show the operations given in the proposition respect
the Poisson boundary and coboundary, which can be found in, for example,
\cite[Chapter 3]{LGPV}.
\end{proof}

In the above proposition, let $L:=[d, \iota]$. Then similar to the Hochschild
case,
$\partial=L_{\pi}$, where $\pi$ is the Poisson structure.

For a Poisson algebra $A$, suppose there exists a form $\eta\in\Omega^n(A)$ such that
$\iota_{(-)}\eta:\mathfrak X^\bullet (A)\to\Omega^{n+\bullet}(A)$ is an isomorphism, 
then $\eta$ is called a {\it volume form}.
In this case, we have the following diagram
\begin{equation}\label{diag:unimodularPoisson}
\xymatrixcolsep{4pc}
\xymatrix{
\mathfrak X^\bullet(A)\ar[r]^-{\iota_{(-)}\eta}&\Omega^{n+\bullet}(A)\\
\mathfrak X^{\bullet-1}(A)\ar[r]^-{\iota_{(-)}\eta}\ar[u]_{\delta}&\Omega^{n+\bullet-1}(A),\ar[u]_{\partial}
}
\end{equation}
which may not be commutative, i.e., $\eta$ may not be a Poisson cycle.
If there exists a volume form $\eta$ such that the above diagram commutes, then we say $A$
is {\it unimodular}.

\begin{theorem}\label{thm:differentialcalculusstructureofunimodularPoisson}
Suppose $A$ is a unimodular Poisson algebra. Then 
$$
(\mathrm{HP}^\bullet(A), \mathrm{HP}_\bullet(A), \wedge,\iota,[-,-], d)
$$
forms a differential calculus with duality.
\end{theorem}

The proof of this theorem is given in Xu \cite{Xu}, and
in particular,  as a corollary, 
$\mathrm{HP}^\bullet(A)$ is a Batalin-Vilkovisky algebra, where the Batalin-Vilkovisky
operator generates the Schouten-Nijenhuis bracket.
The negative cyclic homology of the mixed complex
$(\Omega^\bullet(A), \partial, d)$, 
denoted by $\mathrm{PC}^{-}_\bullet(A)$ and called
the {\it negative cyclic Poisson homology} of $A$,
has a gravity algebra structure, induced from the Batalin-Vilkovisky structure
on the Poisson cohomology.

\subsubsection{Unimodular Frobenius Poisson algebras}

Suppose $A$ is a Poisson algebra, and let $A^*$ be its linear dual space.
For any $P\in\mathfrak X^{p}(A)$ and $\phi\in\mathfrak X^{q}(A; A^*)$,
let
$
\iota_P(\phi)\in\mathfrak X^{p+q}(A; A^*)
$
be given by
\begin{equation}\label{def:innerproductondualvectors}
(\iota_P\phi)(f_1, \cdots, f_{p+q}):=
\sum_{\sigma\in S_{p,q}}\mathrm{sgn}(\sigma)
P(f_{\sigma(1)},\cdots, f_{\sigma(p)})\cdot
\phi(f_{\sigma(p+1)},\cdots, f_{\sigma(p+q)}).
\end{equation}
Observe that
\begin{eqnarray*}
\mathfrak X^\bullet(A; A^*)&=&\mathrm{Hom}_A(\Omega^\bullet(A), A^*)\\
&=&\mathrm{Hom}_A\big(\Omega^\bullet(A), \mathrm{Hom}_k(A, k)\big)\\
&=&\mathrm{Hom}_k(A\otimes_A\Omega^\bullet(A) , k)\\
&=&\mathrm{Hom}_k(\Omega^\bullet(A), k).
\end{eqnarray*}
By dualizing the de Rham differential $d$ on $\Omega^\bullet(A)$,
we obtain a differential $d^*$ on $\mathrm{Hom}(\Omega^\bullet(A), k)$,
i.e., on $\mathfrak X^\bullet(A; A^*)$, which commutes with the Poisson coboundary
(see \cite[Theorem 4.10]{ZVOZ} for a proof).

\begin{proposition}\label{DC-on-Poisson-dual}
Suppose $A$ is a Poisson algebra.
Then 
$$ (\mathrm{HP}^\bullet(A),\mathrm{HP}^\bullet(A; A^*) )$$
has a differential calculus structure, where
$A^*$ is the dual space of $A$.
\end{proposition}

\begin{proof}
Parallel to the proof of Theorem
\ref{dual-Hoch-DC}.
\end{proof}

Now, we go to unimodular Frobenius Poisson algebras, a notion
introduced by Zhu, Van Oystaeyen and Zhang in
\cite{ZVOZ}.
Suppose $A^!$ is a finite dimensional graded Poisson algebra.
If there is an $A^!$-module isomorphism
$$
\eta^!:(A^{!})^\bullet\longrightarrow (A^{\ac})_{n+\bullet},
\quad\mbox{for some\;} n\in\mathbb N,
$$
where $A^{\ac}:=(A^{!})^*=\mathrm{Hom}(A^!,k)$, then we
may view $\eta^!$ as an element in $\mathrm{Hom}_{A^!}(A^!, A^{\ac})\subset\mathfrak X^\bullet(A^!; A^{\ac})$, and
have a diagram
\begin{equation}\label{formula:unimodularcyclicPoisson}
\xymatrixcolsep{4pc}
\xymatrix{
\mathfrak X^\bullet(A^!) \ar[r]^-{\iota_{(-)}\eta^!}&\mathfrak X^{n+\bullet}(A^!;
A^{\ac})\\
\mathfrak X^{\bullet-1}(A^!)\ar[r]^-{\iota_{(-)}\eta^!}\ar[u]_{\delta}
&\mathfrak X^{n+\bullet-1}(A^!; A^{\ac}).\ar[u]_{\delta}
}
\end{equation}

\begin{definition}[Unimodular Frobenius Poisson algebra; \cite{ZVOZ}]
Suppose $A^!$ is a finite dimensional (graded) Poisson algebra.
If there is an $\eta^!\in\mathfrak X^\bullet({A^!};A^{\ac})$ (also called the {\it volume form})
such that
the digram
\eqref{formula:unimodularcyclicPoisson}
commutes,
then $A^!$ is called a {\it unimodular Frobenius Poisson algebra} of degree $n$.
\end{definition}

Alternatively described, a volume form is a nonzero element $\eta^!$ in the
top degree of $A^{\ac}$. It gives the unimodular Poisson structure on $A^!$ if
and only if it is a Poisson cycle.
From the definition, one immediately deduces that:

\begin{theorem}\label{thm:differentialcalculusstructureofunimodularFrobeniusPoisson}
Suppose $A$ is a unimodular Frobenius Poisson algebra. Then 
$$
(\mathrm{HP}^\bullet(A), \mathrm{HP}^\bullet(A; A^*))
$$
forms a differential calculus with duality.
\end{theorem}

The proof of this theorem is given in \cite{ZVOZ}, although the authors
did not express the statement in the above form.
As a corollary, $\mathrm{HP}^\bullet(A)$ is a Batalin-Vilkovisky algebra, where the Batalin-Vilkovisky
operator again generates
the Schouten-Nijenhuis bracket.
The cyclic cohomology of the mixed cochain complex
$(\mathfrak X^\bullet(A; A^*), \delta, d^*)$,
denoted by $\mathrm{PC}^\bullet(A)$ and called the {\it cyclic Poisson
cohomology} of $A$,
therefore has a gravity algebra structure again induced from the Batalin-Vilkovisky structure
on the Poisson cohomology.

\section{Koszul duality}\label{Sect_Koszul}

The purpose of this section is to relate the gravity algebras obtained
in previous sections by means of Koszul duality.
This completes relationships (a) and (b) listed in \S\ref{Sect_intro}.

\subsection{Quadratic and Koszul algebras}
Let $V$ be a finite-dimensional (possibly graded) vector space over $k$.
Denote by $TV$ the free algebra generated by $V$ over $k$.
Suppose $R$ is a subspace of $V\otimes V$, and let
$(R)$ be the two-sided ideal generated by $R$ in $TV$,
then the quotient algebra
$A:= TV/(R)$
is called
a {\it quadratic algebra}.
There are two concepts associated to a quadratic algebra, namely,
its {\it Koszul dual coalgebra} and {\it Koszul dual algebra}, which are given
as follows:

(1) Consider the subspace
$$U=\bigoplus_{n=0}^\infty U_n:=
\bigoplus_{n=0}^\infty \bigcap_{i+j+2=n}V^{\otimes i}\otimes R\otimes
V^{\otimes j}$$
of $TV$,
then $U$ is not an algebra, but a coalgebra, whose coproduct is induced from
the de-concatenation of the tensor products.
The {\it Koszul dual coalgebra} of $A$, denoted
by $A^{\ac}$, is
$$
A^{\ac}=\bigoplus_{n=0}^\infty \Sigma^{\otimes n} (U_n),
$$
where $\Sigma$ is the degree shifting-up (suspension) functor.
$A^{\ac}$ naturally has a graded coalgebra structure induced from that of $U$;
for example, if all elements of $V$ have degree zero,
then
$$
(A^{\ac})_0=k, \quad (A^{\ac})_1=V, \quad (A^{\ac})_2=R,\quad\cdots\cdots
$$

(2) The {\it Koszul dual algebra} of $A$, denoted by $A^!$,
is just the linear dual space of $A^{\ac}$, which is then a graded algebra.
More precisely,
Let $V^*=\mathrm{Hom}(V, k)$ be the linear dual space of $V$,
and let $R^\perp$ denote
the space of annihilators of $R$ in $V^*\otimes V^*$.
Shift the grading of $V^*$ down by one, denoted by $\Sigma^{-1}V^*$,
then
$$
A^!=T(\Sigma^{-1}V^*)/((\Sigma^{-1}\otimes\Sigma^{-1})\circ R^{\perp}).
$$

Choose a set of basis $\{e_i\}$ for $V$, and let $\{e_i^*\}$ be their duals in $V^*$.
There is a natural chain complex associated to $A$, called the {\it Koszul complex}:
\begin{equation}\label{Koszul_complex}
\xymatrix{
\cdots\ar[r]^-{\delta}&
A\otimes A^{\ac}_{i+1}\ar[r]^-{\delta}&
A\otimes A^{\ac}_{i}\ar[r]^-{\delta}&
\cdots\ar[r]&
A\otimes A^{\ac}_0\ar[r]^-{\delta}& k,
}
\end{equation}
where for any $r\otimes f\in A\otimes A^{\ac}$,
$\delta(r\otimes f)=\displaystyle\sum_i e_ir\otimes\Sigma^{-1}e_i^*f$.

\begin{definition}[Koszul algebra]
A quadratic algebra $A=TV/(R)$ is called {\it Koszul}
if the Koszul chain complex \eqref{Koszul_complex} is acyclic.
\end{definition}

\subsection{Koszul duality for Calabi-Yau algebras}

For Koszul Calabi-Yau algebras, we have the following result due to Van den Bergh
\cite{VdB}:
Suppose $A$ is a Koszul algebra, and denote by $A^!$ its Koszul dual algebra;
then $A$ is $d$-Calabi-Yau if and only if $A^!$ is cyclic of degree $d$.
This can be seen as follows:
Since $A$ is Koszul, the following complex
$$
\cdots\longrightarrow A\otimes A^{\ac}_m\otimes A\stackrel{b}{\longrightarrow} A\otimes A^{\ac}_{m-1}\otimes A
\stackrel{b}\longrightarrow\cdots\stackrel{b}\longrightarrow
A\otimes A^{\ac}_0\otimes A$$
with 
\begin{equation}\label{twosidedboundaryonthehochschildcpx}
b(a\otimes c\otimes a')=
\sum_i \big(e_ia\otimes  e_i^* c\otimes a'+(-1)^{m}a\otimes c e_i^* \otimes  a' e_i\big),
\end{equation}
gives a free, minimal resolution of $A$ as a (left) $A^e$ module.
Now suppose $A$ is Calabi-Yau, we have
\begin{eqnarray*}
\mathrm{RHom}_{A^e}(A, A\otimes A)
&=&\mathrm{Hom}_{A^e}(A\otimes A^{\ac}_\bullet\otimes A,
A\otimes A)\\
&=&A\otimes A^!_\bullet \otimes A
\end{eqnarray*}
in $D(A^e)$, where
the differential of $A\otimes A^!_\bullet\otimes A$
is similar to \eqref{twosidedboundaryonthehochschildcpx}. 
It is isomorphic to $\Sigma^{-n}A$ in $D(A^e)$ means
$$A\otimes A^!_\bullet\otimes A\cong A\otimes\Sigma^{-n} A^{\ac}_\bullet\otimes A,$$
which then implies $A^!\cong\Sigma^{-n} A^{\ac}$ as $A^!$ bimodules by
the uniqueness of minimal resolutions.
That is, $A^!$ is a symmetric Frobenius algebra of dimension $n$.

\begin{proposition}
\label{thm:isoofdifferentialcalculuswithdualityforKoszulCYalgebras}
For a Koszul Calabi-Yau algebra $A$, denote
$A^!$ to be its Koszul dual algebra. Then:
\begin{enumerate}
\item[$(1)$] there exists a quasi-isomorphism of
DG algebras 
$\bar{\mathrm C}^\bullet(A)\simeq \bar{\mathrm
C}^\bullet(A^!)$, which induces, on the Hochschild cohomology level,
an isomorphism of graded commutative algebras;

\item[$(2)$] $(\bar{\mathrm C}_\bullet(A), b, B)\simeq (\bar{\mathrm
C}^\bullet(A^!; A^\ac), \delta, B^*)$ as mixed complexes, which, on the Hochschild homology level, 
maps the volume class
to the volume class;

\item[$(3)$] the DG algebra action (the inner product) of 
$\bar{\mathrm C}^\bullet(A)$ on 
$\bar{\mathrm C}_\bullet(A)$ and that of
$\bar{\mathrm C}^\bullet(A^!)$ on 
$\bar{\mathrm C}^\bullet(A^!; A^{\ac})$
are compatible on the homology level.
\end{enumerate}
\end{proposition}

Since the Gerstenhaber bracket on the Hochschild cohomology
is generated by the Batalin-Vilkovisky operator,
this proposition implies that
the two pairs
$$
(\mathrm{HH}^\bullet(A), \mathrm{HH}_\bullet(A))
\quad\mbox{and}\quad
(\mathrm{HH}^\bullet(A^!), \mathrm{HH}^\bullet(A^!; A^{\ac}))$$
are isomorphic as differential calculus with duality.
Thus as a corollary, the first author together with Yang and Zhou proved
in \cite{CYZ}
that for Koszul Calabi-Yau algebras, $\mathrm{HH}^\bullet(A)\cong\mathrm{HH}^\bullet(A^{\ac})$
as Batalin-Vilkovisky algebras.
For reader's convenience
we sketch the proof of the above proposition
in the following while leaving the details for the interested readers to refer to \cite{CYZ}.

\begin{proof}[Proof of Proposition
\ref{thm:isoofdifferentialcalculuswithdualityforKoszulCYalgebras}]
(1)
First, since $A$ is Koszul, its bar construction $\mathrm{B}(A)\simeq A^{\ac}$ as DG coalgebras.
With this quasi-isomorphism we obtain a quasi-isomorphism of DG algebras
\begin{equation}\label{qisHochschildcochaicomplex}
\bar{\mathrm C}^\bullet(A)\cong \mathrm{Hom}(\mathrm B(A), A)
\simeq\mathrm{Hom}(A^{\ac}, A)\cong A\otimes A^!,
\end{equation}
where the differential on $A\otimes A^!$ is given by
\begin{equation*}
\delta(a\otimes x)=\sum_i \Big(
e_i a \otimes e_i^* x  +(-1)^{|x|}  a  e_i  \otimes xe_i^*\Big).
\end{equation*}
Second, with the same argument and with the differentials appropriately assigned, 
we obtain a quasi-isomorphism of
DG algebras
\begin{equation}\label{qisHochschildcochaicomplexofKoszuldual}
\bar{\mathrm C}^\bullet(A^!)\cong \mathrm{Hom}(\mathrm B(A^!), A^!)
\simeq A^!\otimes\Omega(A^{\ac})\cong A^!\otimes A,
\end{equation}
where $\Omega(A^{\ac})$ is the cobar construction of $A^{\ac}$,
which is quasi-isomorphic to $A$ as DG algebras.
Now the right-most two DG algebras
in \eqref{qisHochschildcochaicomplex}
and \eqref{qisHochschildcochaicomplexofKoszuldual}
are quasi-isomorphic as DG algebras via $a\otimes x\mapsto
x\otimes a$, from which follows
the desired quasi-isomorphism 
$\bar{\mathrm C}^\bullet(A)\simeq \bar{\mathrm
C}^\bullet(A^!)$.

(2)
Equip $A\otimes A^{\ac}$  
with differential
\begin{equation*}
b(a\otimes c)=\sum_{i}
\Big(a e_i \otimes c\cdot e_i^*  +(-1)^{m}  e_i a\otimes   e_i^* \cdot c\Big).
\end{equation*}
It is proved in \cite[Lemma 16]{CYZ}
that we have a commutative diagram
\begin{equation}\label{qishochschildchaincpx}
\xymatrix{
&\bar{\mathrm{C}}_\bullet(\Omega(A^\ac))\ar[ld]_{p_1}\ar@<.2ex>[rd]^{p_2}&\\
\bar{\mathrm{C}}_\bullet(A)&&\bar{\mathrm{C}}^\bullet(A^!; A^\ac),
\ar@<.3ex>[lu]^{q_2}\ar[ld]^{\phi_2}\\
&A\otimes A^{\ac}\ar[lu]^{\phi_1}&
}
\end{equation}
where $p_1$ and $p_2$ are quasi-isomorphisms of mixed complexes,
all other maps are quasi-isomorphisms of $b$-complexes,
and $p_2$ and $q_2$ are homotopy inverse to each other.
This means that even though $A\otimes A^{\ac}$ has no mixed
complex structure, on the homology level, it gives an isomorphism
$$
\mathrm{HH}_\bullet(A)\cong\mathrm H_\bullet(A\otimes A^{\ac}, b)
\cong\mathrm{HH}^\bullet(A^!; A^{\ac})
$$
which identifies the $B$ operator on the left-most term with the $B^*$
operator on the right-most term.

It is also proved in \cite{CYZ} that
the volume class in $\mathrm{HH}_\bullet(A)$
and $\mathrm{HH}^\bullet(A^!; A^{\ac})$, via the above isomorphism,
is represented by a nonzero element $\omega\in A^{\ac}_n
\cong k\otimes A^{\ac}_n\subset A\otimes A^{\ac}$.

(3)
Since $A^{\ac}$ is a graded coalgebra and hence admits an action of $A^!$, we 
have that $A\otimes A^!$ acts on $A\otimes A^{\ac}$. Denote this action by $\circ$, then
we have a commutative diagram
\begin{equation}\label{modulecompatibilityofKoszulcpx}
\xymatrixcolsep{4pc}
\xymatrix
{
\bar{\mathrm C}^\bullet(A)\ar[d]^{\simeq}\ar@{~>}[r]^-{\cap}&\bar{\mathrm C}_\bullet(A)\\
A\otimes A^!\ar@{~>}[r]^-{\circ}&A\otimes A^{\ac}\ar[u]_{\simeq},
}
\end{equation}
where the curved arrows mean the module actions.
By exactly the same argument, we have that the following diagram
\begin{equation}\label{modulecompatibilityofKoszuldualcpx}
\xymatrixcolsep{4pc}
\xymatrix
{
\bar{\mathrm C}^\bullet(A^!)\ar[d]^{\simeq}\ar@{~>}[r]^-{\cap^*}&\bar{\mathrm C}^\bullet(A^!, A^{\ac})
\ar[d]^{\simeq}\\
A\otimes A^!\ar@{~>}[r]^-{\circ}&A\otimes A^{\ac},
}
\end{equation}
is commutative.
Combining the above two diagrams \eqref{modulecompatibilityofKoszulcpx}
and \eqref{modulecompatibilityofKoszuldualcpx} with
\eqref{qisHochschildcochaicomplex}, \eqref{qisHochschildcochaicomplexofKoszuldual} and \eqref{qishochschildchaincpx},
we obtain (3).
\end{proof}

\begin{example}[Polynomials]\label{Ex:polynomial}
Let $A=k[x_1, x_2,\cdots, x_n]$ be the space of polynomials (the symmetric tensor algebra),
with each $x_i$ having degree zero.
It is well-known ({\it cf.} \cite{LV}) that
$A$ is a Koszul algebra, and its Koszul dual algebra $A^!$ is
the graded symmetric algebra $\mathbf \Lambda(\xi_1,\xi_2,\cdots,\xi_n)$, with grading
$|\xi_i|=-1$.
There is a non-degenerate symmetric pairing on $A^!$ given by
$$(\alpha, \beta)\mapsto \alpha\wedge\beta/\xi_1\cdots\xi_n$$
so that $A^!$ is symmetric Frobenius, and therefore $A$ is
Calabi-Yau. Equivalently, the pairing on $A^!$ that gives $A^\bullet\cong
\Sigma^{-n}A^{\ac}$ is the same as capping with the following form
$$A^\bullet\to
\Sigma^{-n}A^{\ac}: \quad x\mapsto x\cap\xi_1^*\cdots\xi_n^*.$$
By pulling $\xi_1^*\cdots\xi_n^*$ via the quasiisomorphism
$A\otimes A^{\ac}\simeq \bar{\mathrm{C}}_\bullet(A)$, we get the
volume form on $A$. Its homology class is exactly $dx_1\cdots dx_n$ under
the Hochschild-Kostant-Rosenberg map.
\end{example}

\begin{corollary}\label{cor:isoofLieinftyforKCY}
Suppose $A$ is a $d$-Calabi-Yau algebra and $A^!$ is its Koszul dual.
Then
$$
\mathrm{HC}_\bullet^{-}(A)\cong\mathrm{HC}^\bullet(A^!)
$$
as gravity algebras.
\end{corollary}

\begin{proof}
Proposition \ref{thm:isoofdifferentialcalculuswithdualityforKoszulCYalgebras}
implies that
$$(\mathrm{CH}_\bullet(A), b, B)\simeq
(\mathrm{CH}^\bullet(A^{!}; A^{\ac}), \delta, B^*)
$$ 
as mixed complexes. Combine it with the quasi-isomorphism
of the Hochschild cochain complexes, we obtain that 
$$
(\mathrm{HH}^\bullet(A),\mathrm{HH}_\bullet(A))
\quad\mbox{and}\quad
(\mathrm{HH}^\bullet(A^!),\mathrm{HH}^\bullet(A^!; A^{\ac}))
$$
are isomorphic as differential calculus with duality, which,
by
Theorems \ref{thm:differentialcalculuswithdualityforCY}
and
\ref{thm:differentialcalculuswithdualityforFrob},
induces an isomorphism
$$
\mathrm{HH}_\bullet(A)\cong\mathrm{HH}^\bullet(A^!; A^{\ac})
$$
as Batalin-Vilkovisky algebras.
The statement now follows from Theorem \ref{mainthm} (2).
\end{proof}

\subsection{Koszul duality for quadratic Poisson polynomial algebras}

For $A=k[x_1,x_2,\cdots, x_n]$ with a quadratic bivector
\begin{equation}\label{formula:quadratic_Poisson}
\pi=\sum_{i_1,i_2,j_1,j_2}c_{i_1i_2}^{j_1j_2}x_{i_1}x_{i_2}
\frac{\partial}{\partial x_{j_1}}\wedge\frac{\partial}{\partial x_{j_2}},\quad
c_{i_1i_2}^{j_1j_2}\in k,
\end{equation}
there is a bivector $\pi^!$
on $A^!=\mathbf \Lambda(\xi_1,\xi_2,\cdots,\xi_n)$, 
which we would call the Koszul dual of $\pi$ and is given by
\begin{equation}\label{corresp:PP}
\pi^{!}:=\sum_{i_1,i_2,j_1,j_2}c_{i_1i_2}^{j_1j_2}\xi_{j_1}\xi_{j_2}\frac{\partial}{\partial \xi_{i_1}}\wedge\frac{\partial}{\partial \xi_{i_2}}.
\end{equation}
Shoikhet showed in \cite{Shoikhet} that
for $A=k[x_1,\cdots,x_n]$ with a bivector $\pi$ in the form \eqref{formula:quadratic_Poisson}.
Then
$(A, \pi)$ is Poisson if and only if
$
(A^!,\pi^!)
$
is Poisson.
We have the following result, obtained in \cite{CCEY}:

\begin{proposition}\label{thm:iso_BatalinVilkovisky}
Suppose $A=k[x_1,\cdots,x_n]$ is a quadratic Poisson algebra, and let $A^!$ be its Koszul dual.
Then $A$ is unimodular if and only if $A^{!}$ is unimodular Frobenius,
and in this case we have ($``\cong"$ means {\it isomorphism}):
\begin{enumerate}
\item[$(1)$] $\mathrm{CP}^\bullet(A)\cong\mathrm{CP}^\bullet(A^!)$
as DG algebras;

\item[$(2)$] $\mathrm{CP}_\bullet(A)\cong\mathrm{CP}^\bullet(A^!; A^{\ac})$
as mixed complexes, and moreover, under the isomorphism, the volume
form of the former is mapped to the volume form of the latter;

\item[$(3)$] The DG algebra action (the inner product) of $\mathrm{CP}^\bullet(A)$
on $\mathrm{CP}_\bullet(A)$ and that of $\mathrm{CP}^\bullet(A^!)$
on $\mathrm{CP}^\bullet(A^!; A^{\ac})$ are compatible under the above two isomorphisms.
\end{enumerate}
\end{proposition}

\begin{proof}
(1) 
Since $A=k[x_1,\cdots,x_n]$, we have an explicit expression
for $\Omega^\bullet(A)$, which is
\begin{equation}\label{formula:Poissonchaincpx}
\Omega^\bullet(A)=\mathbf\Lambda(x_1,\cdots, x_n, dx_1,\cdots, dx_n),
\end{equation}
where $\mathbf\Lambda$ means the graded symmetric tensor product, and
$|x_i|=0$ and $|dx_i|=1$, for $i=1,\cdots, n$.
Similarly,
\begin{equation}\label{formula:PoissonchaincpxofKoszuldual}
\Omega^\bullet(A^!)=\mathbf\Lambda(\xi_1,\cdots,\xi_n,d\xi_1,\cdots, d\xi_n),
\end{equation}
where
$|\xi_i|=-1$ and $|d\xi_i|=0$ for $i=1,\cdots, n$.
From \eqref{formula:Poissonchaincpx} and
\eqref{formula:PoissonchaincpxofKoszuldual}
we have the following:
\begin{eqnarray}
\mathfrak X^\bullet(A)&=&\mathrm{Hom}_A(\Omega^\bullet(A),A)\nonumber\\
&=&\mathrm{Hom}_{\mathbf\Lambda(x_1,\cdots,x_n)}
(\mathbf\Lambda(x_1,\cdots,x_n,dx_1,\cdots,dx_n), \mathbf\Lambda(x_1,\cdots,x_n))
\nonumber\\
&=&\mathrm{Hom}_{\mathbf\Lambda(x_1,\cdots,x_n)}(\mathbf\Lambda(x_1,\cdots,x_n)\otimes
\mathbf\Lambda(dx_1,\cdots,dx_n), \mathbf\Lambda(x_1,\cdots,x_n))\nonumber\\
&=&\mathrm{Hom}(\mathbf\Lambda(dx_1,\cdots,dx_n), \mathbf\Lambda(x_1,\cdots,x_n))\nonumber\\
&=&\mathbf\Lambda\Big(x_1,\cdots,x_n,
\frac{\partial}{\partial x_1},\cdots,\frac{\partial}{\partial x_n}\Big)\nonumber
\end{eqnarray}
and
\begin{eqnarray}
\mathfrak X^\bullet(A^!)&=&\mathrm{Hom}_{A^!}(\Omega^\bullet(A^!),A^!)\nonumber\\
&=&\mathrm{Hom}_{\mathbf\Lambda(\xi_1,\cdots,\xi_n)}
(\mathbf\Lambda(\xi_1,\cdots,\xi_n,d\xi_1,\cdots,d\xi_n), \mathbf\Lambda(\xi_1,\cdots,\xi_n))\nonumber\\
&=&\mathrm{Hom}_{\mathbf\Lambda(\xi_1,\cdots,\xi_n)}
(\mathbf\Lambda(\xi_1,\cdots,\xi_n)\otimes\mathbf\Lambda
(d\xi_1,\cdots,d\xi_n), \mathbf\Lambda(\xi_1,\cdots,\xi_n))\nonumber\\
&=&\mathrm{Hom}(\mathbf\Lambda
(d\xi_1,\cdots,d\xi_n), \mathbf\Lambda(\xi_1,\cdots,\xi_n))\nonumber\\
&=&\mathbf\Lambda\Big(\xi_1,\cdots,\xi_n,
\frac{\partial}{\partial \xi_1},\cdots,\frac{\partial}{\partial \xi_n}\Big).\nonumber
\end{eqnarray}
Under the identification
\begin{equation}\label{111identificationofPoissonchainandcochaincpxes}
x_i\mapsto\frac{\partial}{\partial \xi_i},
\quad
\frac{\partial}{\partial x_i}\mapsto \xi_i,
\quad
i=1,\cdots, n,
\end{equation}
we get $\mathfrak X^\bullet(A)\cong\mathfrak X^\bullet(A^!)$.
It is now straightforward to check that this isomorphism is in fact an isomorphism of
DG Gerstenhaber algebras.

(2) Similarly,
\begin{eqnarray}
\mathfrak X^\bullet(A^!; A^{\ac})
&=&\mathrm{Hom}_{A^!}(\Omega^\bullet(A^!), A^{\ac})\nonumber\\
&=&\mathrm{Hom}_{\mathbf\Lambda(\xi_1,\cdots,\xi_n)}(\mathbf\Lambda(\xi_1,\cdots,\xi_n,d\xi_1,\cdots, d\xi_n),
\mathrm{Hom}(\mathbf\Lambda(\xi_1,\cdots,\xi_n),k))\nonumber\\
&=&\mathrm{Hom}_{\mathbf\Lambda(\xi_1,\cdots,\xi_n)}(\mathbf\Lambda(\xi_1,\cdots,\xi_n)\otimes\mathbf\Lambda
(d\xi_1,\cdots, d\xi_n),
\mathrm{Hom}(\mathbf\Lambda(\xi_1,\cdots,\xi_n),k))\nonumber\\
&=&\mathrm{Hom}(\mathbf\Lambda
(d\xi_1,\cdots, d\xi_n),
\mathrm{Hom}(\mathbf\Lambda(\xi_1,\cdots,\xi_n),k))\nonumber\\
&=&\mathrm{Hom}(\mathbf\Lambda
(d\xi_1,\cdots, d\xi_n)\otimes\mathbf\Lambda(\xi_1,\cdots,\xi_n), k)\nonumber\\
&=&\mathrm{Hom}(\mathbf\Lambda
(d\xi_1,\cdots, d\xi_n,\xi_1,\cdots,\xi_n), k)\nonumber\\
&=&\mathbf\Lambda\Big(\xi_1^*,\cdots,\xi_n^*,
\frac{\partial}{\partial \xi_1},\cdots,
\frac{\partial}{\partial\xi_n}
\Big).\nonumber
\end{eqnarray}
Under the identifications
\begin{equation}\label{identificationofPoissonchainandcochaincpxes}
x_i\mapsto\frac{\partial}{\partial \xi_i},
\quad dx_i\mapsto\xi_i^*,
\quad
i=1,\cdots, n,
\end{equation}
we get 
$\Omega^\bullet(A)\cong\mathfrak X^\bullet(A^!; A^{\ac})$.
It is again straightforward to show that this is an isomorphism of mixed complexes. 

Note that the volume form of $A$ is $dx_1dx_2\cdots dx_n$ while the volume form
of $A^!$ is $\xi_1^*\xi_2^*\cdots\xi_n^*$;
under \eqref{111identificationofPoissonchainandcochaincpxes}
and \eqref{identificationofPoissonchainandcochaincpxes},
the former is a Poisson cycle if and only if so is the latter.

(3) With identifications \eqref{111identificationofPoissonchainandcochaincpxes}
and \eqref{identificationofPoissonchainandcochaincpxes},
it is direct to see that the inner products given by \eqref{def:innerproductofvectorfields}
and by \eqref{def:innerproductondualvectors}
are compatible.
\end{proof}

\begin{corollary}\label{cor:isoofLieinftyforKPoisson}
Suppose $A=k[x_1,\cdots,x_n]$ is a unimodular quadratic Poisson algebra, 
and let $A^!$ be its Koszul dual.
Then
$$
\mathrm{PC}_\bullet^{-}(A)\cong\mathrm{PC}^\bullet(A^!)
$$
as gravity algebras.
\end{corollary}

\begin{proof} (Compare with the proof of Proposition \ref{cor:isoofLieinftyforKCY})
Proposition \ref{thm:iso_BatalinVilkovisky} 
together with Theorems
\ref{thm:differentialcalculusstructureofunimodularPoisson}
and \ref{thm:differentialcalculusstructureofunimodularFrobeniusPoisson}
shows that
\begin{equation*}
(\mathrm{CP}_\bullet(A), \partial, d)\quad\mbox{and}\quad
(\mathrm{CP}^{\bullet}(A^!; A^{\ac}),\delta, d^*),
\end{equation*}
satisfy the conditions of Theorem \ref{mainthm} (2),
from which the conclusion follows.
\end{proof}

\section{Deformation quantization}\label{sect:deformation}

The purpose of this section is to relate the gravity algebras obtained
in previous sections by means of deformation quantization.
This completes relationships (c) and (d) listed in \S\ref{Sect_intro}.


In the following, we work over $k[\![\hbar]\!]$, where $\hbar$ is a formal parameter.
Recall that for a Poisson algebra $A$ with bracket $\{-,-\}$,
its {\it deformation quantization}, denoted by $A_{\hbar}$,
is a (completed) $k[\![\hbar]\!]$-linear associative product 
(called the {\it star-product}) on $A[\![\hbar]\!]$:
$$
a\ast b=a\cdot b+\mu_1(a,b)\hbar+\mu_2(a,b)\hbar^2+\cdots,
$$
where $\hbar$ is the formal parameter and $\mu_i$ are bilinear operators,
satisfying
$$
\lim_{\hbar\to 0}\frac{1}{\hbar}\left(a\ast b-b\ast a\right)=\{a,b\},\quad\mbox{for all}\; a, b\in A.
$$
In other words,
a deformation quantization of $A$
is a formal quantization of $A$ in the direction of the Poisson bracket.

\subsection{Deformation quantization of Calabi-Yau Poisson algebras}

From now on, $A=k[x_1,\cdots, x_n]$. By Example \ref{Ex:polynomial}, it
is a Calabi-Yau algebra of dimension $n$ with volume class given
by $dx_1 dx_2\cdots dx_n$.

Let $\mu\in\bar{\mathrm C}^{2}(A[\![\hbar]\!])$ be the Hochschild coboundary.
For any $\tilde\mu\in\hbar\cdot\bar{\mathrm C}^{2}(A[\![\hbar]\!])$,
$$\mu+\tilde\mu: A[\![\hbar]\!]\otimes_{k[\![\hbar]\!]} A[\![\hbar]\!]\to A[\![\hbar]\!]$$
defines a new product 
if and only if $\tilde\mu$ is a Maurer-Cartan element, namely 
$$\mu(\tilde\mu)+\frac{1}{2}[\tilde\mu,\tilde\mu]=0.$$
Kontsevich proved in \cite{Kontsevich} that
there is an $L_\infty$-quasi-isomorphism
\begin{equation}\label{Kontsevichsquasiisomorphism}
\mathfrak X^\bullet(A[\![\hbar]\!])\to \bar{\mathrm C}^\bullet(A[\![\hbar]\!])
\end{equation}
between these two DG 
Lie algebras, 
whose first term is the classical Hochschild-Kostant-Rosenberg map.
Here the differential of the former is zero.
Thus as a corollary, up to gauge equivalences,
the set of Maurer-Cartan elements of 
$\mathfrak X^{2}(A[\![\hbar]\!])$,
which is exactly the set of Poisson structures on $A[\![\hbar]\!]$,
is in one-to-one correspondence to
the set of Maurer-Cartan elements of
$\bar{\mathrm C}^{2}(A[\![\hbar]\!])$, which is exactly the set of 
deformation quantizations of $A[\![\hbar]\!]$.

Now suppose $A$ is Poisson with Poisson structure $\pi$. We equip with
$A[\![\hbar]\!]$ the Poisson structure 
$\hbar\pi$. Let $\mu+\tilde\mu\in\bar{\mathrm C}^{2}(A[\![\hbar]\!])$ 
be the corresponding deformed product, where $\tilde\mu\in \hbar\cdot
\bar{\mathrm C}^{2}(A[\![\hbar]\!])$.
Then the Cyclic Formality
Conjecture for chains, proved by Willwacher in 
\cite[Theorem 1.3 and Corollary 1.4]{Willwacher}, says that
\begin{equation}\label{cyclicformalityconjectureforchains}
(\Omega^\bullet(A[\![\hbar]\!]), L_{\hbar\pi}, d)\simeq(\mathrm C_\bullet(A[\![\hbar]\!]),
L_{\mu+\tilde\mu},B)
\end{equation}
is a quasi-isomorphism of mixed complexes.
Moreover, Dolgushev proved in \cite{Dolgushev} the following:

\begin{theorem}[Dolgushev]\label{Dolgushevsresult}
For a Calabi-Yau Poisson algebra $A$, 
its deformation quantization,
denoted by $A_{\hbar}$, is Calabi-Yau over $k[\![\hbar]\!]$ if and only if $A$ is
unimodular.
\end{theorem}

\begin{proof}
See Dolgushev \cite[Theorem 3]{Dolgushev} (an alternative proof by using
differential graded Lie algebras can be found in \cite[(1.3)]{dTdVVdB}.
\end{proof}

\begin{corollary}\label{cor:isoofLieinftyforPoisson}
Suppose $A$ is a unimodular Poisson Calabi-Yau algebra.
Denote by $A_\hbar$ Kontsevich's deformation quantization
of $A$.
Then
$$
\mathrm{PC}_\bullet^{-}(A[\![\hbar]\!])\cong
\mathrm{HC}_\bullet^{-}(A_\hbar)
$$ 
as gravity algebras.
\end{corollary}

\begin{proof}
First, from \eqref{Kontsevichsquasiisomorphism}, we obtain
an $L_\infty$-quasi-isomorphism
$$
(\mathrm{CP}^\bullet(A[\![\hbar]\!]), \delta_{\hbar\pi})
\simeq
(\bar{\mathrm C}^{\bullet}(A[\![\hbar]\!]), \delta_{\mu+\tilde \mu}).
$$
Later it is proved by Manchon and Torossian
in \cite[Th\'eor\`eme 1.2]{MT} that the above is also a quasi-isomorphism 
of DG algebras.

Second, just to repeat \eqref{cyclicformalityconjectureforchains} we have that
\begin{equation*}
(\Omega^\bullet(A[\![\hbar]\!]), L_{\hbar\pi}, d)\simeq(\bar{\mathrm C}_\bullet(A[\![\hbar]\!]),
L_{\mu+\tilde\mu},B)
\end{equation*}
is a quasi-isomorphism of mixed complexes.

Third, it is proved by Calaque and Rossi in \cite[Theorem 6.1]{CR}
that there is a commutative diagram
\begin{equation}\label{Liemoduleactionsarecompatible}
\xymatrixcolsep{4pc}
\xymatrix{
\mathfrak X^\bullet(A[\![\hbar]\!])\ar@{~>}[r]\ar[d]_{\simeq}&\Omega^\bullet(A[\![\hbar]\!])\\
\bar{\mathrm{C}}^\bullet(A[\![\hbar]\!])\ar@{~>}[r]
&\bar{\mathrm{C}}_\bullet(A[\![\hbar]\!]),\ar[u]^{\simeq}
}
\end{equation}
where the curved arrows mean the DG algebra action.

Thus combining the above three results, we have that on the homology level,
$$
(\mathrm{HP}^\bullet(A[\![\hbar]\!]),\mathrm{HP}_\bullet(A[\![\hbar]\!])) \quad\mbox{and}\quad
(\mathrm{HH}^\bullet(A_\hbar),\mathrm{HH}_\bullet(A_\hbar))
$$
are isomorphic as differential calculus. Theorem \ref{Dolgushevsresult} implies
that both pairs are in fact differential calculus with duality. Thus to show 
they are isomorphic as differential calculus with duality, we need to show
that the volume class is mapped to the volume class. However, this is
guaranteed by the Hochschild-Kostant-Rosenberg theorem.
The corollary now follows from Theorem \ref{BVcriterion}.
\end{proof}

\subsection{Deformation quantization of Frobenius Poisson algebras}
Cattaneo and Felder showed in \cite[Appendix]{CF} that
Kontsevich's $L_\infty$-quasi-isomorphism holds also for graded manifolds,
with exactly the same formula. That is, we have
$L_\infty$-quasi-isomorphism
\begin{equation}\label{Kontsevichsquasiisomorphismforgradedmanifolds}
\mathfrak X^\bullet(A^![\![\hbar]\!])\to\bar{\mathrm C}^\bullet(A^![\![\hbar]\!]),
\end{equation}
where $A^!=\mathbf\Lambda(\xi_1,\cdots,\xi_n)$.

Now suppose $A^!$ is Poisson with Poisson structure $\pi^!$. We equip with
$A^![\![\hbar]\!]$ the Poisson structure $\hbar\pi^!$. 
Denote by $\mu^!$ the product of $A^![\![\hbar]\!]$
and let $\mu^!+\tilde\mu^!\in\bar{\mathrm C}^{2}(A^![\![\hbar]\!])$ 
be the corresponding deformed product.
Then the Cyclic Formality
Conjecture for cochains, proposed by Felder and Shoikhet in \cite{FS}
and proved by Willwacher and Calaque in \cite[Theorem 2]{WC}, says that
\begin{equation}\label{cyclicformalityconjectureforcochains}
(\mathfrak X^\bullet(A^![\![\hbar]\!]; A^{\ac}[\![\hbar]\!]), L_{\hbar\pi^!}^*, d^*)
\simeq(\mathrm C^\bullet(A^![\![\hbar]\!]; A^{\ac}[\![\hbar]\!]),
L_{\mu^!+\tilde\mu^!}^*,B^*)
\end{equation}
is a quasi-isomorphism of mixed complexes.
Moreover, Willwacher-Calaque also showed in \cite{WC} that:

\begin{theorem}[Willwacher-Calaque]\label{thm:WC}
Suppose $A^{!}=\mathbf\Lambda(\xi_1,\cdots,\xi_n)$ is a Frobenius Poisson algebra,
then Kontsevich's deformation quantization of $A^{!}$, say $A^{!}_{\hbar}$,
is symmetric Frobenius if and only if $A^{!}$ is unimodular.
\end{theorem}

\begin{proof}
See Willwacher-Calaque \cite[Theorem 37]{WC}.
\end{proof}



\begin{corollary}\label{cor:isoofLieinftyforFrobPoisson}
Suppose $A^!=\Lambda(\xi_1,\cdots,\xi_n)$ is unimodular Frobenius Poisson.
Denote by $A_\hbar^!$ Kontsevich's deformation quantization
of $A^!$.
Then
$$
\mathrm{PC}^\bullet(A^![\![\hbar]\!])\cong
\mathrm{HC}^\bullet(A^!_\hbar)
$$ 
as gravity algebras.
\end{corollary}

\begin{proof}
(Compare with Corollary \ref{cor:isoofLieinftyforPoisson}).
Manchon and Torossian's result says that the $L_\infty$-quasi-isomorphism
\eqref{Kontsevichsquasiisomorphismforgradedmanifolds}
is also a quasi-isomorphism of DG algebras. To relate
with \eqref{cyclicformalityconjectureforcochains}, let us note that 
$$\mathfrak X^\bullet(A^![\![\hbar]\!]; A^{\ac}[\![\hbar]\!])
\quad\mbox{and}\quad
\Omega^\bullet(A^![\![\hbar]\!])$$
are linear dual to each other,
and so are
$$
\bar{\mathrm{C}}^\bullet(A^![\![\hbar]\!]; A^{\ac}[\![\hbar]\!])
\quad\mbox{and}\quad
\bar{\mathrm{C}}_\bullet(A^![\![\hbar]\!]).
$$
Thus in \eqref{Liemoduleactionsarecompatible}, replacing $A[\![\hbar]\!]$
with $A^![\![\hbar]\!]$
and then considering the associated adjoint action,
we obtain the following commutative diagram 
\begin{equation*}
\xymatrixcolsep{4pc}
\xymatrix{
\mathfrak X^\bullet(A^![\![\hbar]\!])\ar@{~>}[r]\ar[d]_{\simeq}
&\mathfrak X^\bullet(A^![\![\hbar]\!]; A^{\ac}[\![\hbar]\!])
\ar[d]_{\simeq}\\
\bar{\mathrm{C}}^\bullet(A^![\![\hbar]\!])\ar@{~>}[r]
&\bar{\mathrm{C}}^\bullet(A^![\![\hbar]\!]; A^{\ac}[\![\hbar]\!]).
}
\end{equation*}

Thus on the homology level, we obtain that
$$
(\mathrm{HP}^\bullet(A^![\![\hbar]\!]),\mathrm{HP}^\bullet(A^![\![\hbar]\!];
A^{\ac}[\![\hbar]\!])) \quad\mbox{and}\quad
(\mathrm{HH}^\bullet(A^!_\hbar),\mathrm{HH}^\bullet(A^!_\hbar; A^{\ac}_\hbar))
$$
are isomorphic as differential calculus.
Now Theorem \ref{thm:WC}
further implies that they are quasi-isomorphic as differential calculus with duality.
The corollary now follows from Theorem \ref{BVcriterion}.
\end{proof}

\subsection{Deformation quantization of quadratic unimodular Poisson algebras}

As explained by Shoikhet in \cite[\S1]{Shoikhet},
the Koszul duality theory can be extended
to algebras over $k[\![\hbar]\!]$, which is still valid.
Now, for 
$A=k[x_1,\cdots,x_n]$
and $A^!=\mathbf\Lambda(\xi_1,\cdots,\xi_n)$,
it is shown by Shoikhet in \cite[Theorem 7.1]{Shoikhet}
(see also \cite[Theorem 8.6]{CFFR}) that if $A$ and $A^!$ 
are Koszul dual to each other as Poisson algebras, then their 
deformation quantization
$A_{\hbar}$ and $A^!_{\hbar}$ are also Koszul dual to each other 
as associative algebras over $k[\![\hbar]\!]$.
Therefore,
$A$ being unimodular implies that $A_\hbar$ and $A^!_\hbar$ are
Koszul dual as Calabi-Yau and symmetric Frobenius algebras.
We have shown
in Corollaries \ref{cor:isoofLieinftyforKCY} and 
\ref{cor:isoofLieinftyforKPoisson}
that
\begin{equation*}
\mathrm{HC}_\bullet^{-}(A_{\hbar})\cong\mathrm{HC}^\bullet(A^!_{\hbar})
\quad\mbox{and}
\quad\mathrm{PC}_\bullet^{-}(A[\![\hbar]\!])\cong\mathrm{PC}^\bullet(A^![\![\hbar]\!])
\end{equation*}
as gravity algebras by Koszul duality,
and in Corollaries \ref{cor:isoofLieinftyforPoisson} 
and \ref{cor:isoofLieinftyforFrobPoisson}
that
\begin{equation*}
\mathrm{PC}_\bullet^{-}(A[\![\hbar]\!])\cong
\mathrm{HC}_\bullet^{-}(A_\hbar)\quad\mbox{and}\quad
\mathrm{PC}^\bullet(A^![\![\hbar]\!])\cong
\mathrm{HC}^\bullet(A^!_\hbar)
\end{equation*}
as gravity algebras via deformation quantization.
Combining these results
we have the following statement, which also verifies the commutative diagram 
\eqref{deformationquantizationofKoszulunimodular}
in \S\ref{Sect_intro}:

\begin{theorem}
Let $A=k[x_1,\cdots, x_n]$ be a quadratic unimodular Poisson algebra.
Denote by $A^{\ac}$ the Koszul dual Poisson algebra of $A$,
and by $A_\hbar$ and $A_\hbar^{!}$
Kontsevich's deformation quantization of $A$ and $A^{\ac}$ respectively.
Then
the following
$$
\xymatrixcolsep{4pc}
\xymatrix{
\mathrm{PC}_{\bullet}^{-}(A[\![\hbar]\!])\ar[d]\ar[r]&
\mathrm{PC}^\bullet(A^{!}[\![\hbar]\!])\ar[d]
\\
\mathrm{HC}_{\bullet}^{-}(A_\hbar)\ar[r]&
\mathrm{HC}^\bullet(A^{!}_\hbar),
}
$$
where
$\hbar$ be a formal parameter,
is a commutative diagram 
of isomorphisms of gravity algebras.
\end{theorem}


\end{document}